\documentclass{mcom-l}
\usepackage{graphicx}
\usepackage{subfigure}

\newtheorem{theorem}{Theorem}[section]

\theoremstyle{definition}

\theoremstyle{remark}
\newtheorem{remark}[theorem]{Remark}

\usepackage{epstopdf} 
\usepackage{color}   

\numberwithin{equation}{section}

\begin{document}

\title{Gaussian quadrature rules for composite highly oscillatory integrals}


\author{Menghan Wu}
\address{School of Mathematics and Statistics, Huazhong University of Science and Technology, Wuhan 430074, P. R. China}
\email{menghanwu@hust.edu.cn}
\thanks{}

\author{Haiyong Wang}
\address{School of Mathematics and Statistics, Huazhong University of Science and Technology, Wuhan 430074, P. R. China}
\address{Hubei Key Laboratory od Engineering Modeling and Scientific Computing, Huazhong University of Science and Technology, Wuhan 430074, P. R.
China} \curraddr{} \email{haiyongwang@hust.edu.cn}
\thanks{This work was supported by National
Natural Science Foundation of China under Grant number 11671160.}


\date{}

\dedicatory{}

\begin{abstract}
Highly oscillatory integrals of composite type arise in electronic
engineering and their calculations are a challenging problem. In
this paper, we propose two Gaussian quadrature rules for computing
such integrals. The first one is constructed based on the classical
theory of orthogonal polynomials and its nodes and weights can be
computed efficiently by using tools of numerical linear algebra. An
interesting connection between the quadrature nodes and the Legendre
points is proved and it is shown that the rate of convergence of
this rule depends solely on the regularity of the non-oscillatory
part of the integrand. The second one is constructed with respect to
a sign-changing function and the classical theory of Gaussian
quadrature can not be used anymore. We explore theoretical
properties of this Gaussian quadrature, including the trajectories
of the quadrature nodes and the convergence rate of these nodes to
the endpoints of the integration interval, and prove its asymptotic
error estimate under suitable hypotheses. Numerical experiments are
presented to demonstrate the performance of the proposed methods.
\end{abstract}

\maketitle

\section{Introduction}
Highly oscillatory integrals arise in a wide range of practical
applications, such as acoustic and electromagnetic scattering,
optics, electronic circuits and fluid mechanics. The computation of
such integrals with traditional quadrature rules like Newton-Cotes
and Gauss quadrature becomes intractable due to the oscillatory
feature of the integrands. In the recent two decades, highly
oscillatory integrals of the Fourier-type, i.e.,
\begin{equation}\label{eq:Fourier}
I[f] = \int_{a}^{b} f(x) e^{i\omega h(x)} \mathrm{d}x,  \quad
\omega\gg1,
\end{equation}
where $i$ is the imaginary unit and $f$ and $h$ are referred to as
the {\it amplitude} and {\it phase}, have attracted substantial
attention and several effective methods, such as asymptotic and
Filon-type methods, Levin-type methods, numerical steepest descent
methods and complex Gaussian quadrature, have been developed  (see,
e.g.,
\cite{Asheim2014,huybrechs2006,Iserles2005,Levin1982,Olver2006}).
All these methods share a remarkable advantage that their accuracy
improves rapidly as $\omega$ increases. We refer the interested
readers to the recent monograph \cite{DHIbook2017} for a
comprehensive survey.

In the numerical simulation of electronic circuits, the governing
equation for the diode rectifier circuit can be formulated as the
following nonlinear ordinary differential equation
\begin{equation}\label{ODE}
C \frac{\mathrm{d}v(t)}{\mathrm{d}t} = I_0 \left[ e^{b(t)-v(t)}-1
\right] - \frac{v(t)}{R}, \quad t\geq0, \quad v(0)=v_0,
\end{equation}
where $C$ is the capacitance, $R$ is the resistance, $I_0$ is the
diode reverse bias current, $b(t)$ is the input signal and the
unknown $v(t)$ is the voltage (see, e.g.,
\cite{Condon2009a,Condon2009b,Dautb2005,Iserles2011}). When choosing
the input signal $b(t)=\kappa\sin(\omega t)$ or
$b(t)=\kappa\cos(\omega t)$ and using waveform relaxation to
\eqref{ODE}, it is necessary to compute integrals of the forms
\begin{align}\label{eq:ExpSin}
\int_{a}^{b} f(x) e^{\kappa\sin(\omega x)} \mathrm{d}x, \quad
\int_{a}^{b} f(x) e^{\kappa\cos(\omega x)} \mathrm{d}x.
\end{align}
It is easily seen that the above integrals are highly oscillatory
integrals of composite type, which do not fit into the pattern of
highly oscillatory integrals of Fourier-type in \eqref{eq:Fourier}.
Indeed, the above composite highly oscillatory integrals behave like
$O(1)$ as $\omega\rightarrow\infty$ \cite{Iserles2011} and the
Fourier-type integrals in \eqref{eq:Fourier} behave like
$O(\omega^{-\alpha})$ for some $\alpha>0$, and thus the asymptotic
behaviors of \eqref{eq:ExpSin} and \eqref{eq:Fourier} are quite
different.

In this paper, we are concerned with the composite highly
oscillatory integrals of the form
\begin{equation}\label{CHOI}
I[f] := \int_{a}^{b} f(x) (g\circ \phi_{\omega})(x) \mathrm{d}x,
\quad  \omega\gg1,
\end{equation}
where $g:[-1,1]\rightarrow\mathbb{R}$ is a nonnegative and smooth
function and $\phi_{\omega}(x)=\sin(\omega x)$ or
$\phi_{\omega}(x)=\cos(\omega x)$. Note that \eqref{CHOI} includes
\eqref{eq:ExpSin} as a special case (i.e., $g(x)=e^{\kappa x}$). For
such integrals, to the best of our knowledge, only few works have
been conducted in the literature on their asymptotic expansion and
numerical computations. We mention the asymptotic and Filon-type
methods developed in \cite{Iserles2011} and the nonlinear
approximation methods developed in \cite{Beylkin2016}. However, each
of these methods has its own disadvantage, such as the error of the
asymptotic method is uncontrollable for a fixed $\omega$ and the
Filon-type methods require the computation of higher order
derivatives of $g(x)$ which is a notoriously ill-conditioned problem
(see \cite[Section~2.2]{Iserles2011}). As for the nonlinear
approximation method, it is based on the nonlinear approximation to
$f(x)$ by sums of complex exponentials and then evaluating the
resulting integral with generalized Gaussian quadrature, which is
relatively expensive in applications.

The goal of this paper is to develop two Gaussian quadrature rules
for computing the composite highly oscillatory integrals
\eqref{CHOI}. The first one is guaranteed by the standard theory of
orthogonal polynomials and it is optimal in the sense that an
$n$-point quadrature rule integrates exactly whenever
$f\in\mathcal{P}_{2n-1}$, where $\mathcal{P}_{k}$ denotes the space
of polynomials of degree at most $k$ (i.e.,
$\mathcal{P}_{k}=\mathrm{span}\{1,x,\ldots,x^k\}$). We show that
this rule is spectrally accurate whenever $f$ is sufficiently
smooth. The second Gaussian quadrature rule is constructed with
respect to a sign-changing function and thus the existence of this
rule can not be guaranteed. We explore numerically several
theoretical aspects of this Gaussian quadrature, including the
trajectories of the quadrature nodes and the convergence rate of
these nodes to the endpoints of the integration interval. Once the
quadrature rule exists, we prove that it is optimal in the sense
that the error of an $n$-point quadrature rule behaves like
$O(\omega^{-n-1})$ as $\omega\rightarrow\infty$, and thus its
accuracy improves rapidly as $\omega$ increases.

The rest of the paper is organized as follows. In section
\ref{sec:FirstGauss} we propose the first Gaussian quadrature rule
for computing \eqref{CHOI}. We provide a detailed description and
present a rigorous convergence analysis of the quadrature rule. In
section \ref{sec:SecondGauss}, we propose the second Gaussian
quadrature rule for computing \eqref{CHOI} based on the asymptotic
analysis of the integrals \eqref{CHOI}. Finally, in section
\ref{sec:conclusion}, we present some conclusions of our study.

\section{The first Gaussian quadrature rule}\label{sec:FirstGauss}
In this section we construct the first Gaussian quadrature rule for
computing the composite highly oscillatory integrals \eqref{CHOI}.
The key idea is to treat the composite function $(g\circ
\phi_{\omega})(x)$ as the weight function and construct a quadrature
rule of the form
\begin{equation}\label{eq:FirstGauss}
I[f] = \int_{a}^{b} f(x) (g\circ \phi_{\omega})(x) \mathrm{d}x =
\sum_{k=1}^n w_k f(x_k) + R_n[f],
\end{equation}
such that $R_n[f]=0$ whenever $f\in \mathcal{P}_{2n-1}$. More
precisely, let $\{p_n^{\omega}(x)\}_{n=0}^{\infty}$ be a sequence of
polynomials orthogonal with respect to $(g\circ \phi_{\omega})(x)$.
Notice that if $g(x)$ is smooth and nonnegative, the existence of
the sequence of the orthogonal polynomials
$\{p_n^{\omega}(x)\}_{n=0}^{\infty}$ is always guaranteed. Moreover,
from the standard theory of Gaussian quadrature rules we know that
the nodes $\{x_k\}_{k=1}^{n}$ are precisely the zeros of
$p_n^{\omega}(x)$ and they are all located inside the interval
$(a,b)$ and the quadrature weights $\{w_k\}_{k=1}^{n}$ are all
positive (see, e.g., \cite{Gautschi2004}).

In the following subsections, we shall discuss the computational
aspects and develop some sharp error bounds for the Gaussian
quadrature rule \eqref{eq:FirstGauss}. Hereafter, we shall restrict
our attention to the case of $[a,b]=[-1,1]$ for the sake of
simplicity. However, the generalization to a more general setting is
mathematically straightforward.
\subsection{Polynomials orthogonal with respect to $(g\circ \phi_{\omega})(x)$}\label{sec:OrthPolyI}
Let $\{T_0(x),T_1(x),\ldots\}$ be the sequence of the Chebyshev
polynomials of the first kind, i.e., $T_k(x)=\cos(k\arccos(x))$. We
express $p_n^{\omega}(x)$ in terms of Chebyshev polynomials as
\begin{equation}\label{def:pn}
p_n^{\omega}(x) = T_n(x) + \sum_{k=0}^{n-1} a_k T_k(x).
\end{equation}
From the orthogonality of $p_n^{\omega}(x)$ it follows that
\begin{equation}\label{eq:OrthCond}
\int_{-1}^1 p_n^{\omega}(x) T_j(x) (g\circ \phi_{\omega})(x)
\mathrm{d}x = 0,  \quad  j=0,\ldots,n-1,
\end{equation}
or equivalently,
\begin{equation}\label{eq:LinearEqn}
\left(
  \begin{array}{ccccc}
    \mu_{0,0} & \mu_{1,0} & {\cdots}& \mu_{n-1,0}\\
    \mu_{0,1} & \mu_{1,1} & {\cdots}& \mu_{n-1,1}\\
    {\vdots} & {\vdots} & {\ddots}& {\vdots}\\
    \mu_{0,n-1} & \mu_{1,n-1} & {\cdots}& \mu_{n-1,n-1}\\
  \end{array}
\right) \left(
  \begin{array}{ccccc}
    a_0 \\
    a_1 \\
    {\vdots} \\
    a_{n-1} \\
  \end{array}
\right) =- \left(
  \begin{array}{ccccc}
    \mu_{n,0} \\
    \mu_{n,1} \\
    {\vdots} \\
    \mu_{n,n-1} \\
  \end{array}
\right),
\end{equation}
where
\begin{equation}\label{def:Mu}
\mu_{k,j} = \int_{-1}^1 T_j(x) T_k(x) (g\circ \phi_{\omega})(x)
\mathrm{d}x.
\end{equation}
It is easily verified that the matrix on the left-hand side of
\eqref{eq:LinearEqn} is symmetric positive definite and direct
calculations show that its condition number for a fixed $\omega$
behaves like $O(n)$ as $n$ increases, and thus the linear system
\eqref{eq:LinearEqn} can be solved efficiently by standard linear
solvers. After solving \eqref{eq:LinearEqn}, we obtain the
coefficients $\{a_k\}_{k=0}^{n-1}$, and the zeros of
$p_n^{\omega}(x)$ (i.e., $\{x_k\}_{k=1}^{n}$) can be obtained by
computing the eigenvalues of the following colleague matrix (see,
e.g., \cite[Theorem~18.1]{Trefethen2020})
\begin{equation}
C = \left(
  \begin{array}{cccccc}
    0 & 1 \\
    \frac{1}{2}& 0 &\frac{1}{2}   \\
     &\frac{1}{2}&{\ddots} &{\ddots}  \\
     & &{\ddots}&{\ddots}&\frac{1}{2} \\
     & & &\frac{1}{2} &0 \\
  \end{array}
\right) - \frac{1}{2} \left(
  \begin{array}{cccccc}
    0&0 &\cdots&0   \\
    0&0 &\cdots&0   \\
    \vdots&\vdots &\ddots & \vdots  \\
    0&0 &\cdots&0   \\
   a_0&a_1&{\cdots}&a_{n-1} \\
  \end{array}
\right).
\end{equation}
In order to compute the quadrature nodes $\{x_k\}_{k=1}^{n}$, it is
necessary to evaluate $\{\mu_{k,j}\}$ for $k=0,\ldots,n$ and
$j=0,\ldots,n-1$. To achieve this, we introduce the modified
Chebyshev moments
\begin{align}
\nu_j := \int_{-1}^1 T_j(x) (g\circ \phi_{\omega})(x) \mathrm{d}x,
\quad j=0,1,\ldots.
\end{align}
Recalling that $T_j(x)T_k(x)=(T_{j+k}(x)+T_{|j-k|}(x))/2$ for all
$j,k\geq0$, we obtain $\mu_{k,j}=(\nu_{j+k}+\nu_{|j-k|})/2$, and
thus it is sufficient to compute the modified Chebyshev moments
$\{\nu_0,\ldots,\nu_{2n-1}\}$. We now present two approaches to
compute these modified moments:
\begin{itemize}
\item[(I).] If $g(x)$ is analytic inside the disc $|z|<r$ for some $r>1$. Let
\begin{align}\label{def:rho}
\rho_m = \frac{1}{2^{m-1}} \sum_{k=0}^\infty\frac{g^{(m+2k)}(0)}{k!
(m+k)! 4^k}, \quad  m=0,1,\ldots,
\end{align}
and let
\begin{equation}\label{eq:UV}
\begin{array}{l}
U_k(x) = {\displaystyle \sum_{j=1}^\infty\frac{(-1)^j \rho_{2j}}{(2j)^{2k+1}}\sin(2j\omega x) - \sum_{j=0}^{\infty} \frac{(-1)^j \rho_{2j+1}}{(2j+1)^{2k+1}}\cos((2j+1)\omega x)}, \\
V_k(x) = {\displaystyle \sum_{j=1}^{\infty} \frac{(-1)^j
\rho_{2j}}{(2j)^{2k+2}} \cos(2j\omega x) + \sum_{j=0}^{\infty}
\frac{(-1)^j \rho_{2j+1}}{(2j+1)^{2k+2}} \sin((2j+1)\omega x)}.
\end{array}
\end{equation}
From \cite{Iserles2011} we obtain for $\phi_{\omega}(x)=\sin(\omega
x)$ that
\begin{align}\label{eq:nu1}
\nu_j = \frac{\rho_0}{2} &\int_{-1}^{1} T_j(x) \mathrm{d}x + \sum_{k=0}^{\lfloor j/2 \rfloor} \frac{(-1)^k}{\omega^{2k+1}} \left[ T_j^{(2k)}(1)U_k(1) - T_j^{(2k)}(-1)U_k(-1) \right]  \\
&+ \sum_{k=0}^{\lfloor (j-1)/2 \rfloor} \frac{(-1)^k}{\omega^{2k+2}}
\left[T_j^{(2k+1)}(1)V_k(1) - T_j^{(2k+1)}(-1)V_k(-1) \right],
\nonumber
\end{align}
and the derivatives of Chebyshev polynomials at both endpoints can
be calculated by (see, e.g., \cite[Equation (A.13)]{JPBoyd2000})
\[
T_n^{(k)}(\pm1) = (\pm1)^{n+k} \prod_{j=0}^{k-1}
\frac{n^2-j^2}{2j+1}.
\]
Furthermore, it was shown in \cite[Theorem~1]{Iserles2011} that the
sequence $\{\rho_m\}_{m=0}^{\infty}$ decays exponentially fast, and
hence each $U_k(x)$ and $V_k(x)$ can be calculated accurately by
truncating the infinite series properly. We remark that this
approach is also applicable to the case
$\phi_{\omega}(x)=\cos(\omega x)$.

\item[(II).] We approximate the weight function with its Chebyshev expansion, i.e.,
\begin{align}\label{eq:gCheb}
(g\circ\phi_{\omega})(x) \approx \sum_{k=0}^{m} c_{k} T_{k}(x),
\end{align}
and the coefficients $\{c_k\}_{k=0}^{m}$ can be calculated
efficiently by using the fast Fourier transform (FFT) in $O(m\log
m)$ operations, and then
\begin{equation}\label{eq:nu2}
\nu_j \approx \sum_{\substack{k=0 \\ k+j~ \mathrm{even}} }^{m}
c_{k}\left[\frac{1}{1-{(k+j)}^2}+\frac{1}{1-{(k-j)}^2}\right].
\end{equation}
\end{itemize}
We remark that the approach (I) is advantageous whenever $n$ is
small and its computational cost does not increase as $\omega$
increases. However, this approach requires the computation of higher
order derivatives of $g(x)$, which is a notoriously ill-conditioned
problem. In contrast, the approach (II) is advantageous for all
practical purposes since it can be implemented by means of the FFT.
However, the parameter $m$ in \eqref{eq:gCheb} will increase
linearly with $\omega$ to achieve a given target accuracy.

\begin{remark}
We can also apply Clenshaw-Curtis or Gauss-Legendre quadrature to
calculate the modified Chebyshev moments $\{\nu_j\}_{j=0}^{2n-1}$
directly and this strategy is advantageous when $j$ is of small or
moderate size. When $j$ is large, however, a disadvantage of this
strategy is that the number of quadrature nodes may be much larger
than the number $m$ in \eqref{eq:nu2} when a prescribed accuracy for
evaluating $\nu_j$ is desired. This is what can be expected, since
$T_j(x)$ is also highly oscillatory for large $j$.
\end{remark}

Next, we consider the computation of the quadrature weights
$\{{w_k}\}_{k=1}^{n}$. Recalling that the Gaussian quadrature rule
\eqref{eq:FirstGauss} is exact for $f\in\mathcal{P}_{2n-1}$, we
infer that
\begin{equation}\label{eq:QuadW}
\sum_{k=1}^n w_k T_j(x_k) = \nu_j, \quad  j=0,\ldots,n-1,
\end{equation}
or equivalently,
\begin{equation}\label{eq:LineaeEqnW}
\left(
  \begin{array}{cccc}
    T_0(x_1) & T_0(x_2) & {\cdots}& T_0(x_n)\\
    T_1(x_1) & T_1(x_2) & {\cdots}& T_1(x_n)\\
    {\vdots} & {\vdots} & {\ddots}& {\vdots}\\
    T_{n-1}(x_1) & T_{n-1}(x_2) & {\cdots}& T_{n-1}(x_n)\\
  \end{array}
\right) \left(
  \begin{array}{cccc}
    w_1 \\
    w_2 \\
    {\vdots} \\
    w_n \\
  \end{array}
\right) = \left(
  \begin{array}{ccccc}
    \nu_0   \\
    \nu_1   \\
    {\vdots} \\
    \nu_{n-1} \\
  \end{array}
\right),
\end{equation}
Hence, the quadrature weights $\{w_k\}_{k=1}^{n}$ can be obtained
from solving \eqref{eq:LineaeEqnW}.

\begin{figure}
\centering
\includegraphics[width=4.cm,height=4cm]{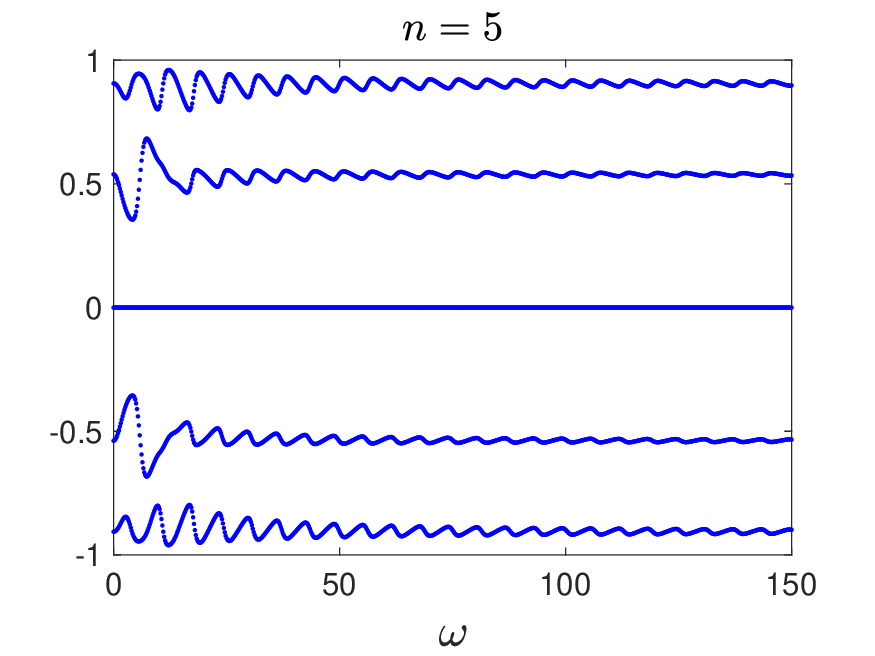}
\includegraphics[width=4.cm,height=4cm]{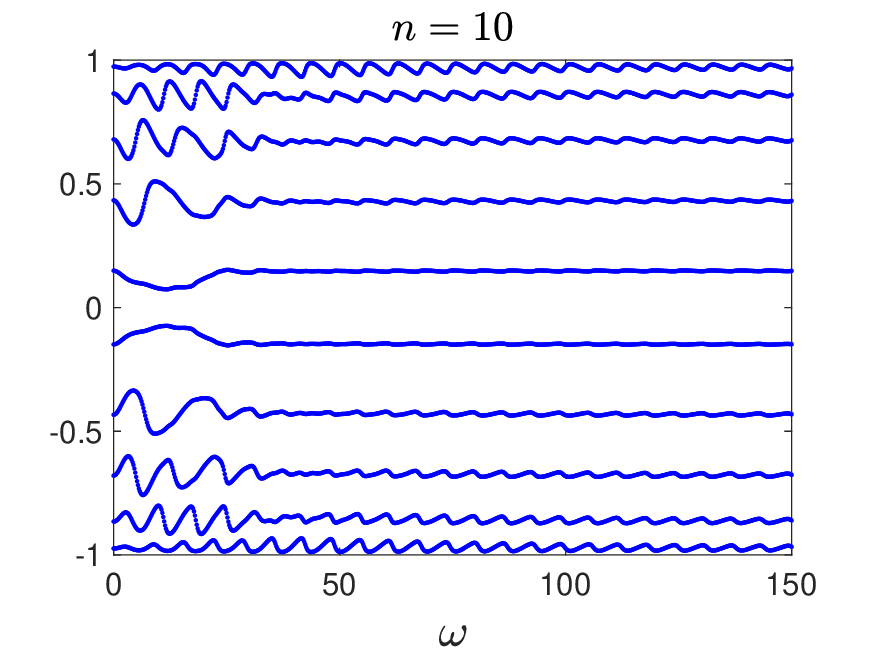}
\includegraphics[width=4.cm,height=4cm]{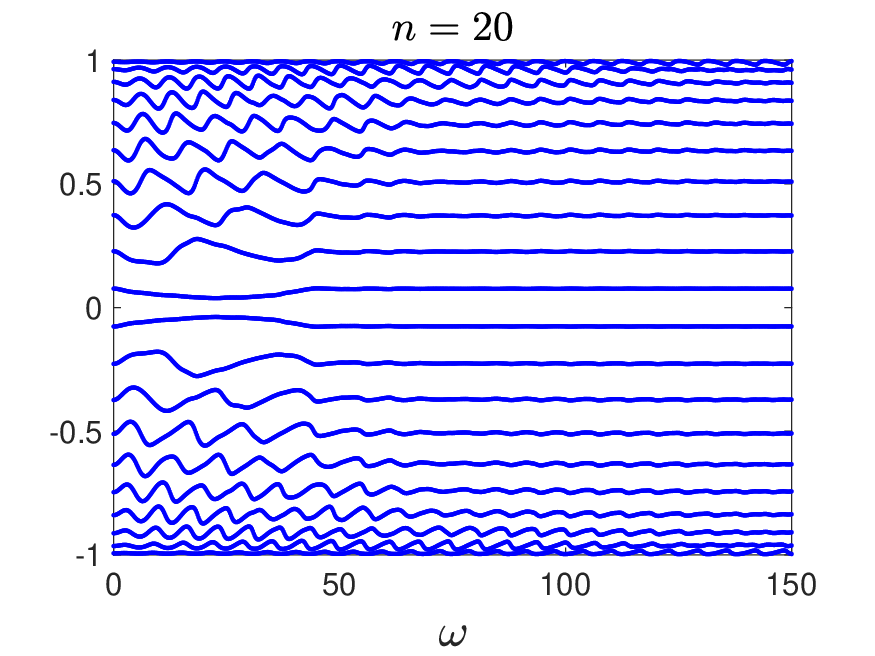}
\caption{The nodes $\{x_k\}_{k=1}^{n}$ as a function of $\omega$ for
$n=5,10,20$.} \label{xklocation1}
\end{figure}

\begin{figure}
\centering
\includegraphics[width=4.cm,height=4cm]{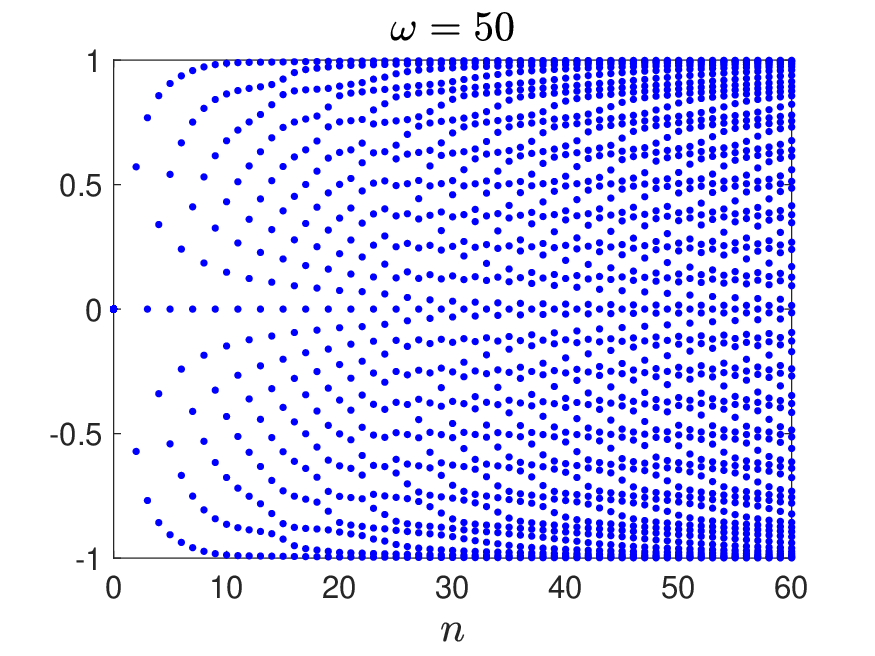}
\includegraphics[width=4.cm,height=4cm]{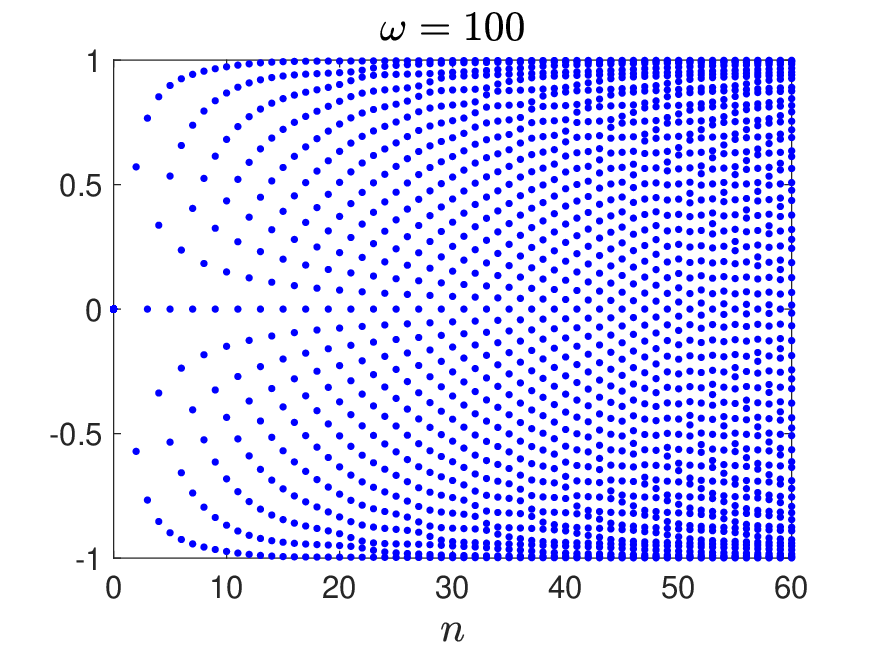}
\includegraphics[width=4.cm,height=4cm]{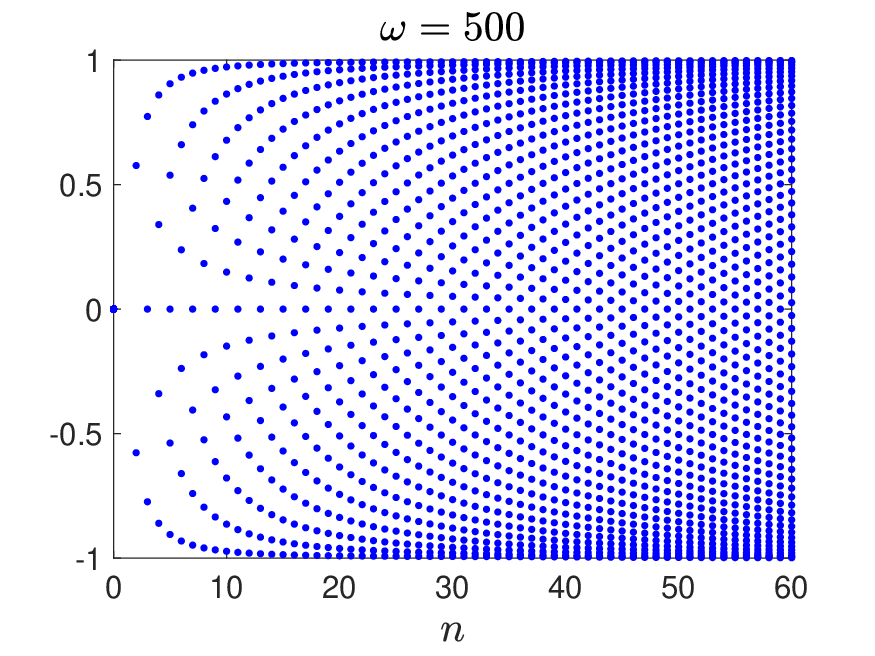}
\caption{The nodes $\{x_k\}_{k=1}^{n}$ as a function of $n$ for
$\omega=50,100,500$.} \label{xklocation2}
\end{figure}

We now present some numerical experiments to illustrate the behavior
of the nodes and weights of the Gaussian quadrature rule and the
computations were performed using {\sc Matlab} on a laptop with an
Intel(R) Core(TM) i5-8265U CPU with 8GB RAM. In all of the
experiments, we apply the approach (II) to compute the modified
moments $\{\nu_0,\ldots,\nu_{2n-1}\}$ and the parameter $m$ is
chosen such that the last Chebyshev coefficient $c_m$ is just above
machine precision. Figure \ref{xklocation1} illustrates the nodes
$\{x_k\}_{k=1}^{n}$ as a function of $\omega$ for three values of
$n$ and Figure \ref{xklocation2} illustrates the nodes
$\{x_k\}_{k=1}^{n}$ as a function of $n$ for three values of
$\omega$ for the weight function
$(g\circ\phi_{\omega})(x)=\exp(2\cos(\omega x))$. From Figure
\ref{xklocation1} we see that each node $x_k$ converges to a
limit value as $\omega\rightarrow\infty$. A natural question to ask
is what these limit values are? Interestingly, numerical experiments
show that these limit values are exactly the zeros of $P_n(x)$,
where $P_n(x)$ is the Legendre polynomial of degree $n$. To show
this, let $\{x_k^{\mathrm{GL}}\}_{k=1}^{n}$ be the zeros of
$P_n(x)$, we plot the absolute errors $|x_k-x_k^{\mathrm{GL}}|$
scaled by $\omega$ in Figure \ref{xklk}. Clearly, we see that the
nodes $\{x_k\}_{k=1}^{n}$ converge to the Legendre points
$\{x_k^{\mathrm{GL}}\}_{k=1}^{n}$ at the rate $O(\omega^{-1})$.
Below we provide a rigorous proof for this observation.
\begin{theorem}\label{thm:ZeroLimit}
Let $\{x_k\}_{k=1}^{n}$ and $\{x_k^{\mathrm{GL}}\}_{k=1}^{n}$ denote
the zeros of $p_n^{\omega}(x)$ and $P_n(x)$, respectively, and
assume that they are ordered according to the same rule. Then, for
each $k\in\{1,\ldots,n\}$, we have
\begin{equation}\label{eq:ZeroLimit}
\left| x_k - x_k^{\mathrm{GL}} \right| = O(\omega^{-1}), \quad
\omega\rightarrow\infty.
\end{equation}
\end{theorem}
\begin{proof}
First, we define the following standard moments
\begin{equation}
\lambda_{k}=\int_{-1}^{1} x^k (g\circ \phi_{\omega})(x) \mathrm{d}x,
\quad k=0,1,\ldots. \nonumber
\end{equation}
As $\omega\rightarrow\infty$, by \cite[Equation~(2.3)]{Iserles2011}
we know that
\begin{equation}
\lambda_{k}=\frac{\rho_{0}}{2}\tau_{k}+ O(\omega^{-1}), \nonumber
\end{equation}
where $\rho_0$ is defined in \eqref{def:rho} and $\tau_{k}=0$ when
$k$ is odd and $\tau_{k}=2/(k+1)$ when $k$ is even. Recalling the
determinantal representation of orthogonal polynomials
\cite[Chapter~1]{Chihara1978} and using the fact that the leading
term of $T_n(x)$ is $2^{n-1}x^{n}$, we get
\begin{equation}
p_{n}^{\omega}(x) = \frac{2^{n-1}}{\Delta_{n}} \left|
  \begin{array}{ccccc}
    \lambda_{0} & \lambda_{1} & {\cdots}& \lambda_{n}\\
    {\vdots} & {\vdots} &  & {\vdots}\\
    \lambda_{n-1} & \lambda_{n} & {\cdots}& \lambda_{2n-1}\\
    1 & x & {\cdots}& x^n
  \end{array}
\right|, \quad  \Delta_{n} = \left|
  \begin{array}{ccccc}
    \lambda_{0} & {\cdots}& \lambda_{n-1}\\
    {\vdots} & {\ddots} & {\vdots}\\
    \lambda_{n-1} & {\cdots}& \lambda_{2n-2}
  \end{array}
\right|. \nonumber
\end{equation}
Note that $\Delta_{n}\neq0$ for all $n\in\mathbb{N}$ since the
sequence of orthogonal polynomials $\{p_{n}^{\omega}\}$ always
exists. In the special case where $(g\circ \phi_{\omega})(x)=1$, it
is clear that $p_{n}^{\omega}(x)$ reduces exactly to $P_n(x)$ up to
a constant factor. Specifically, let $K_n=(2n)!/(2^n(n!)^2)$ be the
leading coefficient of $P_n(x)$, then
\begin{equation}
P_n(x) = \frac{K_{n}}{L_{n}} \left|
  \begin{array}{ccccc}
    \tau_{0} & \tau_{1} & {\cdots}& \tau_{n}\\
    {\vdots} & {\vdots} &  & {\vdots}\\
    \tau_{n-1} & \tau_{n} & {\cdots}& \tau_{2n-1}\\
    1 & x & {\cdots}& x^n
  \end{array}
\right|, \quad  L_{n}= \left|
  \begin{array}{ccccc}
    \tau_{0} & {\cdots}& \tau_{n-1}\\
    {\vdots} & {\ddots} & {\vdots}\\
    \tau_{n-1} & {\cdots}& \tau_{2n-2}
  \end{array}
\right|, \nonumber
\end{equation}
and $L_{n}\neq0$ for all $n\in\mathbb{N}$. Moreover, from the
asymptotic behavior of the standard moments $\lambda_k$ given above,
we can deduce that
\begin{equation}
\Delta_{n} = \left( \frac{\rho_0}{2}\right)^{n}
L_{n}+O(\omega^{-1}), \quad  \omega\rightarrow\infty. \nonumber
\end{equation}
Next, we consider the asymptotic behavior of $p_{n}^{\omega}(x)$ as
$\omega\rightarrow\infty$. Combining the above four results and
using the elementary properties of determinants, we obtain
\begin{align}
p_{n}^{\omega}(x) = \frac{(\rho_{0})^n}{2\Delta_{n}} \left|
  \begin{array}{ccccc}
    \tau_{0} & \tau_{1} & {\cdots}& \tau_{n}\\
    {\vdots} & {\vdots} &  & {\vdots}\\
    \tau_{n-1} & \tau_{n} & {\cdots}& \tau_{2n-1}\\
    1 & x & {\cdots}& x^n
  \end{array}
\right| + O(\omega^{-1}) &= \frac{(\rho_{0})^n
L_{n}}{2K_n\Delta_{n}}P_n(x) +
O(\omega^{-1})  \nonumber \\
&= \frac{2^{n-1}}{K_n} P_n(x) + O(\omega^{-1}). \nonumber
\end{align}
Hence $p_{n}^{\omega}(x)$ converges to a multiple of $P_n(x)$ for
all $x\in[-1,1]$ at the rate $O(\omega^{-1})$ as
$\omega\rightarrow\infty$, and consequently, the zeros of
$p_{n}^{\omega}(x)$ converge to the corresponding zeros of $P_n(x)$
as $\omega\rightarrow\infty$. For each $k\in\{1,\ldots,n\}$, if we
assume that $x_k=x_k^{\mathrm{GL}}+\varepsilon_k(\omega)$ for some
function $\varepsilon_k(\omega)$, then we infer from the above
equation that $\lim_{\omega\to\infty}\varepsilon_k(\omega)=0$, and
therefore
\begin{align}
p_n^{\omega}(x) &= 2^{n-1}\prod_{k=1}^{n}(x-x_k)  \nonumber\\
&= 2^{n-1}\prod_{k=1}^{n}(x-x_k^{\mathrm{GL}})
+2^{n-1}\sum_{k=1}^{n} \varepsilon_k(\omega)
\prod_{j=1,j\neq k}^{n}(x-x_j^{\mathrm{GL}}) + \textrm{HOT} \nonumber \\
&= \frac{2^{n-1}}{K_n}P_n(x) + 2^{n-1}\sum_{k=1}^{n}
\varepsilon_k(\omega) \prod_{j=1,j\neq k}^{n}(x-x_j^{\mathrm{GL}}) +
\textrm{HOT}, \nonumber
\end{align}
where \textrm{HOT} denotes higher order terms, from which it follows
that
\begin{align}
p_n^{\omega}(x_k^{\mathrm{GL}}) &= 2^{n-1} \varepsilon_k(\omega)
\prod_{j=1,j\neq k}^{n} (x_k^{\mathrm{GL}} - x_j^{\mathrm{GL}}) +
\textrm{HOT}. \nonumber
\end{align}
Note that $p_n^{\omega}(x)=2^{n-1}P_n(x)/K_n + O(\omega^{-1})$ holds
true for all $x\in[-1,1]$, it follows immediately that
$p_n^{\omega}(x_k^{\mathrm{GL}})=O(\omega^{-1})$ as
$\omega\rightarrow\infty$. Combining this with the above equation
and noting that the Legendre points
$\{x_k^{\mathrm{GL}}\}_{k=1}^{n}$ are independent of $\omega$, we
conclude that $\varepsilon_k(\omega)= O(\omega^{-1})$. This ends the
proof.
\end{proof}

Regarding the quadrature weights $\{w_k\}_{k=1}^{n}$, numerical
experiments show that each $w_k$ also converges to a limit value as
$\omega\rightarrow\infty$, however, these limit values are no longer
the corresponding Gauss-Legendre quadrature weights
$\{w_k^{\mathrm{GL}}\}_{k=1}^{n}$.

\begin{figure}
\centering
\includegraphics[width=4.cm,height=4cm]{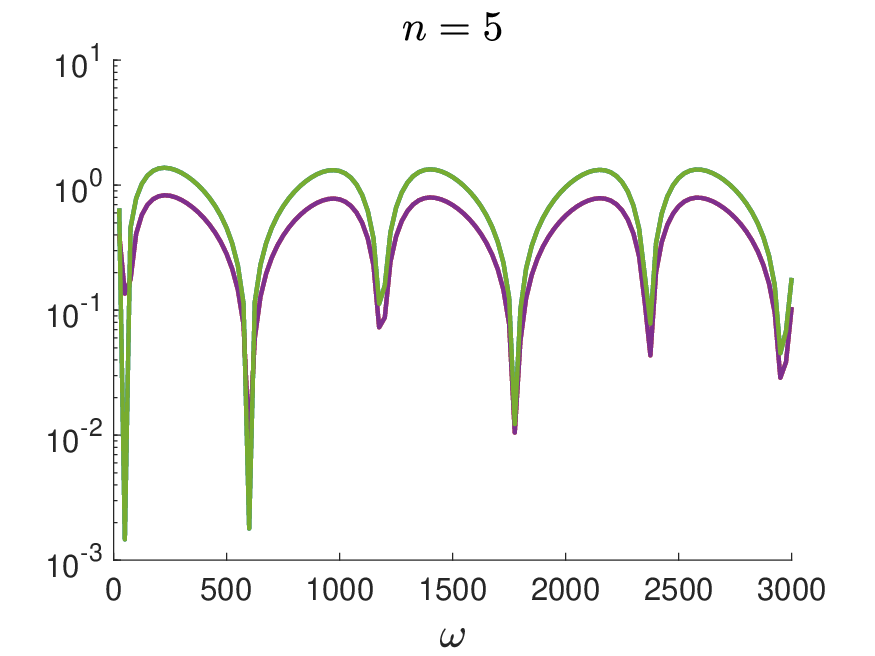}~
\includegraphics[width=4.cm,height=4cm]{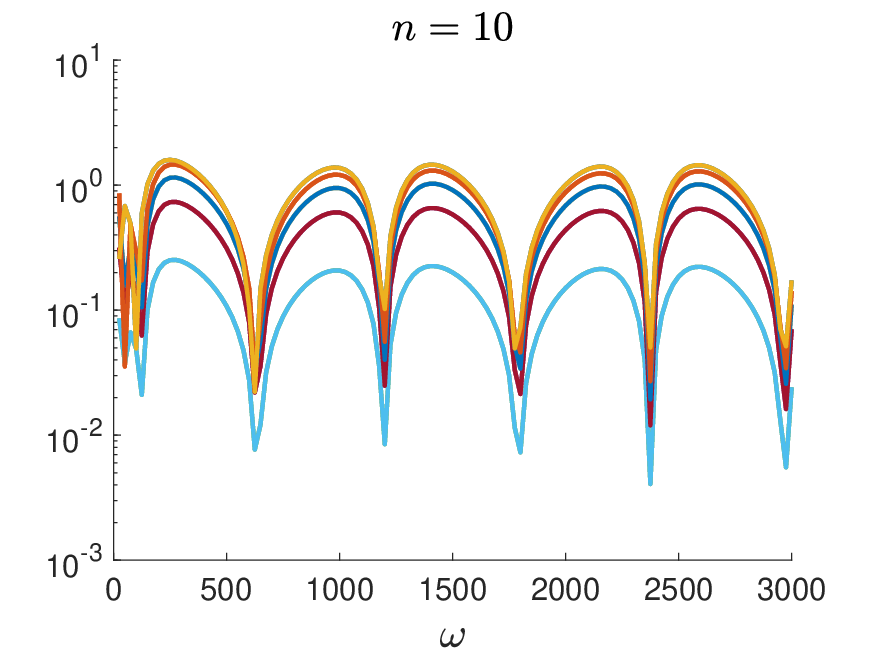}~
\includegraphics[width=4.cm,height=4cm]{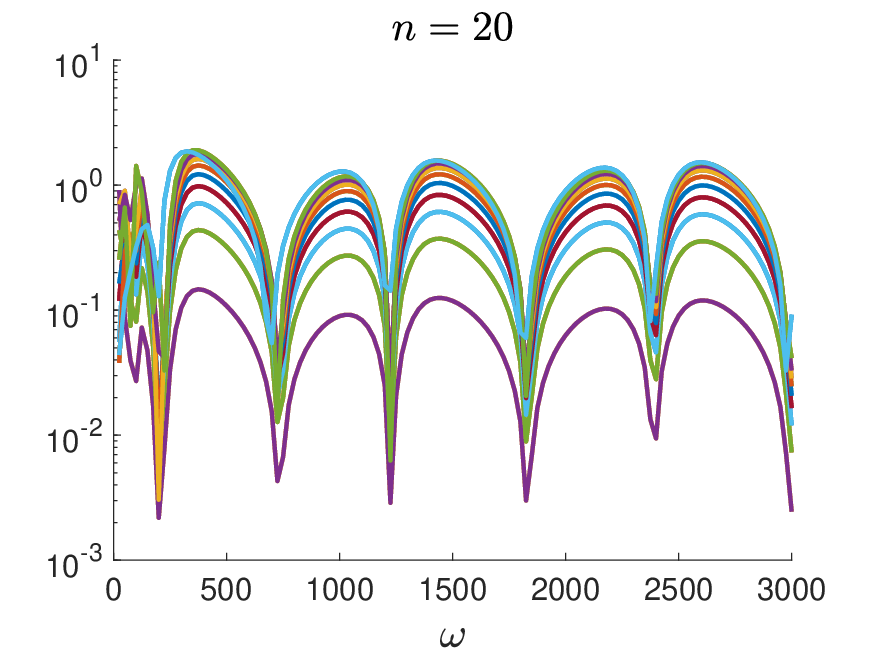}
\caption{The absolute errors $|x_k-x_k^{\mathrm{GL}}|$ with
$k=1,\ldots,n$ scaled by $\omega$ for $n=5,10,20$.} \label{xklk}
\end{figure}

\subsection{Convergence analysis of Gaussian quadrature}
In this subsection we present an error analysis of the Gaussian
quadrature rule defined in \eqref{eq:FirstGauss}. Our main results
are stated in the following theorem.
\begin{theorem}\label{thm:ErrorBound}
Let $R_n[f]$ be the remainder of the Gaussian quadrature rule
defined in \eqref{eq:FirstGauss} and let
$\mathcal{K}_{\omega}=\int_{-1}^{1} (g\circ \phi_{\omega})(x)
\mathrm{d}x$.
\begin{itemize}
\item[\rm (i)] If $f$ is analytic with $|f(z)|\leq M_{\rho}$ in the region bounded by the ellipse with foci $\pm1$ and
major and minor semiaxis lengths summing to $\rho>1$, then for each
$n\geq0$,
\begin{equation}\label{err:anal}
|R_n[f]| \leq \frac{4\mathcal{K}_{\omega}
M_{\rho}}{\rho^{2n}(1-\rho^{-1})}.
\end{equation}

\item[\rm (ii)] If $f,f{'},\ldots,f^{(m-1)}$ are absolutely continuous on $[-1,1]$ and $f^{(m)}$ is of bounded variation $V_m$ for some $m\in\mathbb{N}$. Then, for each $n\geq \lfloor (m+2)/2 \rfloor$,
\begin{equation}\label{err:diff}
|R_n[f]| \leq \frac{4\mathcal{K}_{\omega} V_m}{\pi
m(2n-1)\cdots(2n-m)}.
\end{equation}
\end{itemize}
\end{theorem}
\begin{proof}
We first consider the proof of \eqref{err:anal}. Recalling the
Chebyshev series of $f$, i.e.,
\begin{equation}
f(x) = \frac{a_0}{2} + \sum_{j=1}^\infty a_j T_j(x),  \quad  a_j =
\frac{2}{\pi} \int_{-1}^{1} \frac{f(x) T_k(x)}{\sqrt{1-x^2}}
\mathrm{d}x. \nonumber
\end{equation}
Notice that $R_n[f]=0$ for $f\in\mathcal{P}_{2n-1}$. Substituting
the above Chebyshev series into the remainder yields
\begin{align}\label{eq:Rn}
R_n[f] &= \sum_{j=2n}^{\infty} a_j R_n[T_j] =
\sum_{j=2n}^{\infty}a_j\left( I[T_j]-\sum_{k=1}^{n}w_kT_j(x_k)
\right).
\end{align}
For the term inside the bracket, using the inequality
$|T_j(x)|\leq1$ we find that
\begin{equation}\label{eq:ErrorS1}
\left| I[T_j] - \sum_{k=1}^{n} w_k T_j(x_k) \right| \leq
\mathcal{K}_{\omega} + \sum_{k=1}^{n} w_k = 2 \mathcal{K}_{\omega},
\end{equation}
where we have used the fact that
$\mathcal{K}_{\omega}=\sum_{k=1}^{n}w_k$ in the last equality. On
the other hand, recall that the Chebyshev coefficients of $f$
satisfy $|a_j|\leq2M_{\rho}/\rho^j$ for each $j\geq0$ (see, e.g.,
\cite[Theorem~8.1]{Trefethen2020}). Combining this with
\eqref{eq:ErrorS1} and \eqref{eq:Rn} we obtain that
\begin{align}
|R_n[f]| &\leq \sum_{j=2n}^{\infty}
\frac{4\mathcal{K}_{\omega}M_{\rho}}{\rho^j} =
\frac{4\mathcal{K}_{\omega} M_{\rho}}{\rho^{2n}(1-\rho^{-1})}.
\nonumber
\end{align}
This proves \eqref{err:anal}. Next, we consider the proof of
\eqref{err:diff} and the idea is similar to that of
\eqref{err:anal}. Recall from \cite[Theorem~7.1]{Trefethen2020} that
the Chebyshev coefficients of $f$ satisfy
\begin{equation}
|a_j| \leq \frac{2V_m}{\pi j(j-1)\cdots(j-m)}. \nonumber
\end{equation}
Combining this with \eqref{eq:Rn} and \eqref{eq:ErrorS1} yields
\begin{align}
|R_n[f]| &\leq \sum_{j=2n}^{\infty} \frac{4\mathcal{K}_{\omega}V_m}{\pi j(j-1)\cdots(j-m)} \nonumber \\
&= \frac{4\mathcal{K}_{\omega}V_m}{\pi m} \sum_{j=2n}^{\infty} \left[ \frac{1}{(j-1)\cdots(j-m)} - \frac{1}{j\cdots(j-m+1)}  \right] \nonumber \\
&= \frac{4\mathcal{K}_{\omega} V_m}{\pi m(2n-1)\cdots(2n-m)}.
\nonumber
\end{align}
This proves \eqref{err:diff} and the proof of Theorem
\ref{thm:ErrorBound} is complete.
\end{proof}

Some remarks on Theorem \ref{thm:ErrorBound} are in order.
\begin{remark}
From Theorem \ref{thm:ErrorBound} we can see that the rate of
convergence of the proposed Gaussian quadrature rule in
\eqref{eq:FirstGauss} depends solely on the regularity of $f$. More
precisely, the proposed Gaussian quadrature rule converges at an
exponential rate whenever $f$ is analytic in a neighborhood of
$[-1,1]$ and at an algebraic rate whenever $f$ is differentiable on
$[-1,1]$; see Figure \ref{experiment1} for an illustration.
\end{remark}

\begin{remark}
If the weight function $(g\circ \phi_{\omega})(x)$ is even, then due
to the parity of Chebyshev polynomials we have
\begin{align}\label{eq:RnEven}
R_n[f] &= \sum_{j=0}^{\infty} a_{2n+2j} R_n[T_{2n+2j}],
\end{align}
and hence the error bounds in \eqref{err:anal} and \eqref{err:diff}
can be further improved.
\end{remark}

We now present several numerical experiments to demonstrate the
performance of the Gaussian quadrature rule in
\eqref{eq:FirstGauss}. We consider the following functions
\begin{equation}
f(x) = e^x, \quad  x\sin(x), \quad  \sqrt{x+2}, \quad  (1+x^2)^{-1},
\quad |x|^3, \quad |\sin(x)|,
\end{equation}
and we consider the weight function $(g\circ
\phi_{\omega})(x)=\exp(2\sin(\omega x))$. For these functions of
$f(x)$, it is easily seen that the first two functions are entire
and thus the Gaussian quadrature will converge super-exponentially.
For the third and fourth functions, it is easily seen that they are
analytic in a neighborhood of $[-1,1]$ and thus the Gaussian
quadrature converges exponentially. For the last two test functions,
they are differentiable functions and thus the Gaussian quadrature
converges algebraically. In Figure \ref{experiment1} we plot the
absolute errors of the Gaussian quadrature rule as a function of $n$
for several values of $\omega$. Clearly, we see that numerical
results are consistent with our theoretical analysis.

\begin{figure}
\centering
\includegraphics[width=5.6cm,height=5.cm]{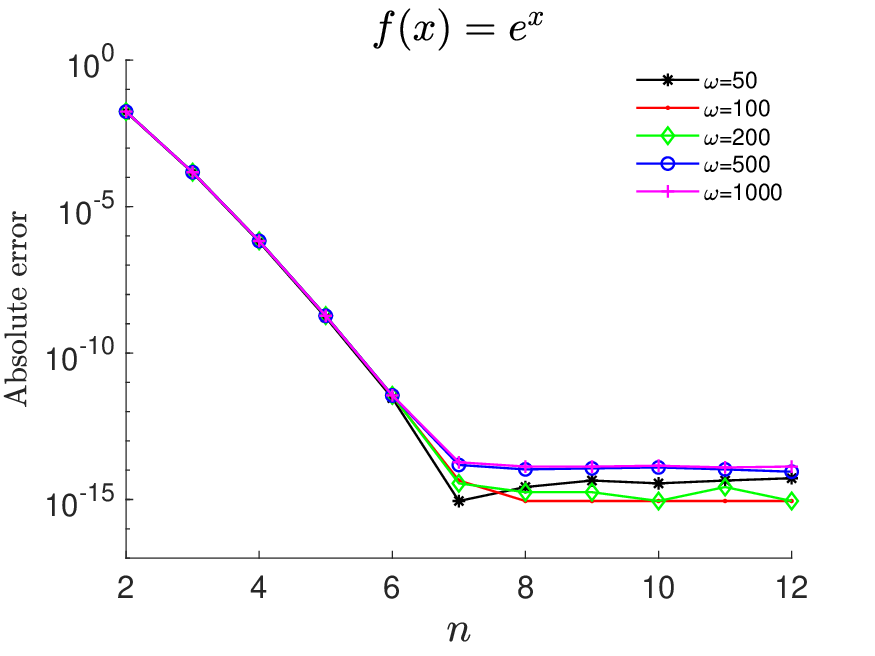}~~~~
\includegraphics[width=5.6cm,height=5.cm]{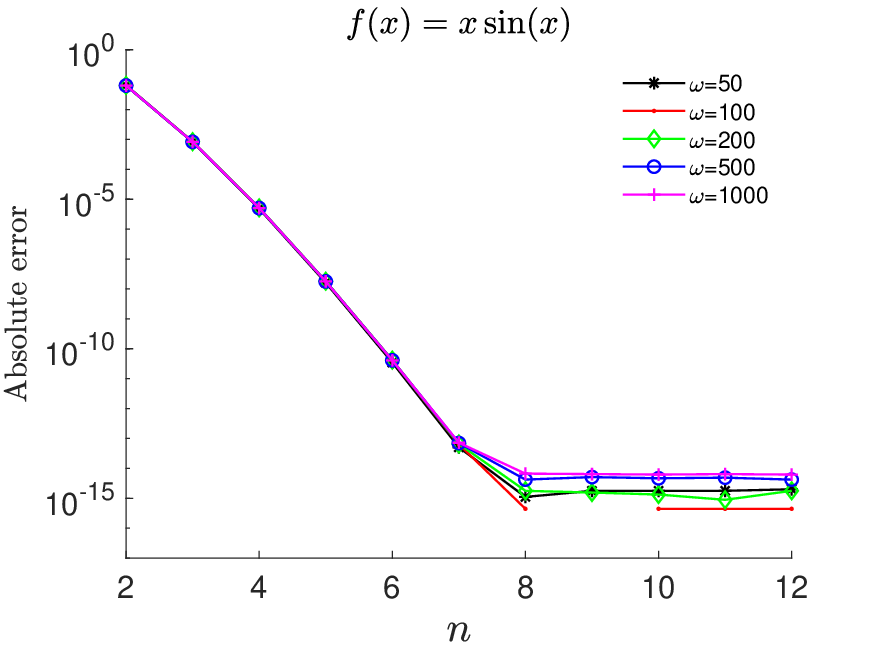}\\
\includegraphics[width=5.6cm,height=5.cm]{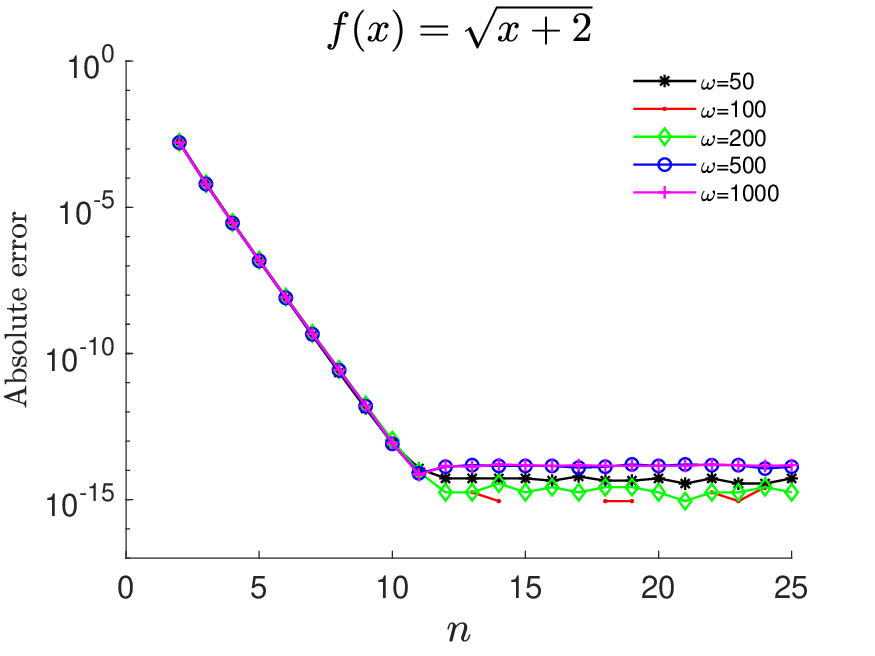}~~~~
\includegraphics[width=5.6cm,height=5.cm]{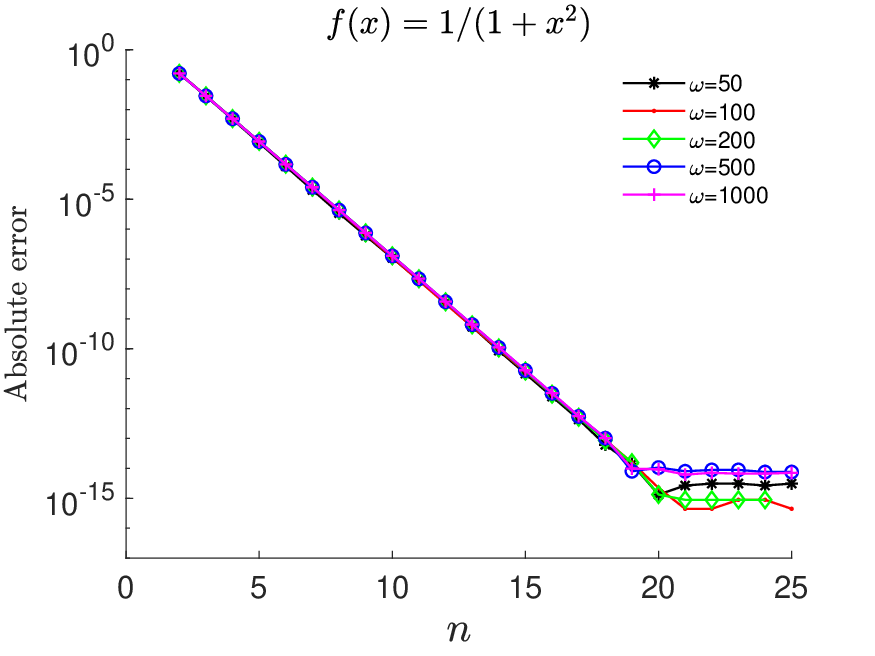}\\
\includegraphics[width=5.6cm,height=5.cm]{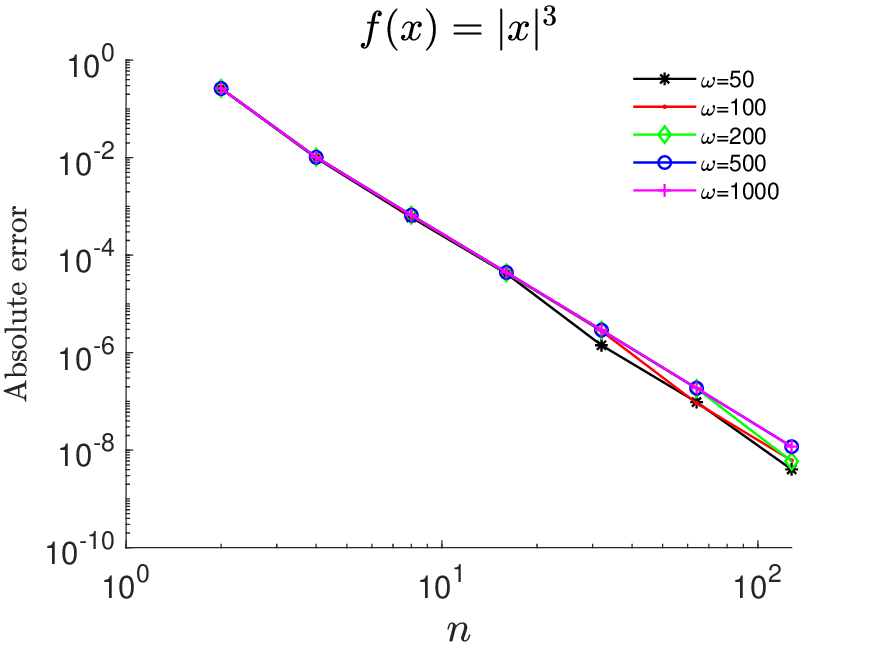}~~~~
\includegraphics[width=5.6cm,height=5.cm]{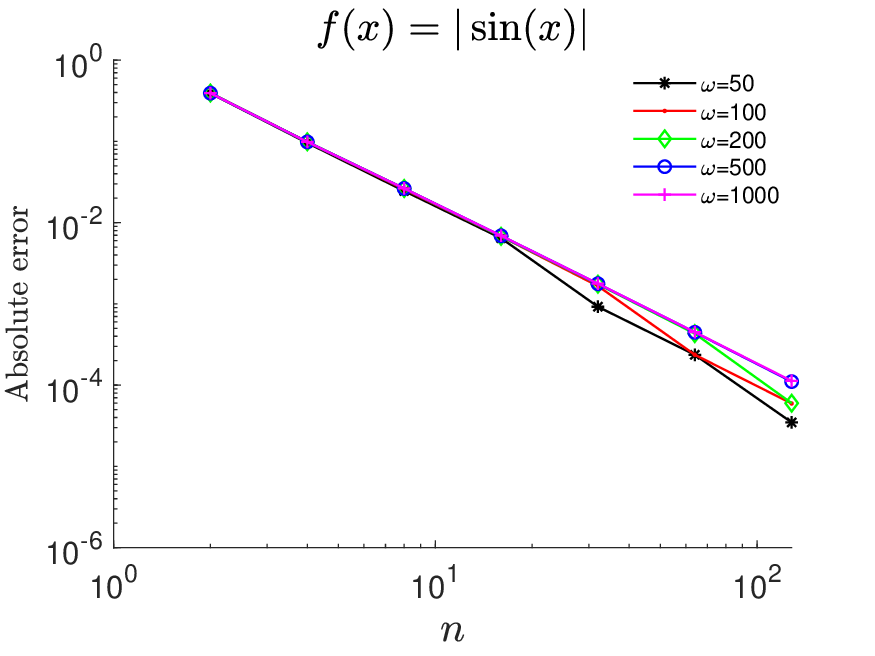}
\caption{Absolute errors of the Gaussian quadrature rule
\eqref{eq:FirstGauss} as a function of $n$ for
$\omega=50,100,200,500,1000$. The first and second rows are plotted
in a log scale and the last row is plotted in a log-log scale.}
\label{experiment1}
\end{figure}

\section{The second Gaussian quadrature rule}\label{sec:SecondGauss}
Although we have constructed an efficient Gaussian quadrature rule
for computing the composite highly oscillatory integrals
\eqref{CHOI}, its accuracy depends solely on the regularity of $f$
and does not improve as the parameter $\omega$ increases. A natural
question to ask is: Can we construct a Gaussian quadrature rule for
computing \eqref{CHOI} such that its accuracy improves as the
parameter $\omega$ increases? In this section we shall consider this
issue and explore an alternative Gaussian quadrature by extending
the idea in \cite{Asheim2014} for computing the standard
Fourier-type integrals to the current setting.

Before proceeding, it is instructive to revisit the asymptotic
expansion of the composite highly oscillatory integrals \eqref{CHOI}
which has been studied in \cite{Iserles2011} in detail. Suppose that
$g(x)$ is analytic inside the disc $|z|<r$ for some $r>1$, it was
shown in \cite[Equation~(2.3)]{Iserles2011} that the integrals
\eqref{CHOI} with $\phi_{\omega}(x)=\sin(\omega x)$ admit the
following asymptotic expansion
\begin{align}\label{eq:asymptotic}
I[f] \sim \frac{\rho_0}{2}\int_a^b f(x) \mathrm{d}x &+ \sum_{k=0}^\infty \frac{(-1)^k}{\omega^{2k+1}} \left[ f^{(2k)}(b)U_k(b)-f^{(2k)}(a)U_k(a) \right]  \\
& + \sum_{k=0}^\infty \frac{(-1)^k}{\omega^{2k+2}}
\left[f^{(2k+1)}(b)V_k(b)-f^{(2k+1)}(a)V_k(a) \right], \nonumber
\end{align}
where the constant $\rho_0$ is defined in \eqref{def:rho} and
$U_k(x)$ and $V_k(x)$ are defined in \eqref{eq:UV}. As a direct
consequence of \eqref{eq:asymptotic} we see immediately that
$I[f]\rightarrow\frac{\rho_0}{2}\int_a^b f(x) \mathrm{d}x$ as
$\omega\rightarrow\infty$. Moreover, it is easily verified that this
asymptotic result still holds whenever $\phi_{\omega}(x)=\cos(\omega
x)$. In the following, we shall explore a new Gaussian quadrature
rule for computing the value of $I[f]$.

We first split the integral $I[f]$ into two parts
\begin{align}\label{eq:Split}
I[f] &= \int_{a}^{b} f(x)\left[(g\circ
\phi_{\omega})(x)-\frac{\rho_0}{2}\right]\mathrm{d}x +
\frac{\rho_0}{2} \int_{a}^{b} f(x) \mathrm{d}x.
\end{align}
From \eqref{eq:asymptotic} we know that the first integral on the
right hand side behaves like $O(\omega^{-1})$ as
$\omega\rightarrow\infty$ and the second integral on the right hand
side is independent of $\omega$. In the sequel, we denote the first
and second terms on the right hand side of \eqref{eq:Split} by
$I_{O}[f]$ and $I_{S}[f]$, respectively. As for $I_{S}[f]$, notice
that its integrand is non-oscillatory, which can be evaluated
efficiently by using Gauss-Legendre or Clenshaw-Curtis quadrature.
Hence, we shall restrict our attention to the computation of
$I_{O}[f]$. Inspired by the Gaussian quadrature rule for computing
the standard Fourier-type integrals, i.e., $\int_{a}^{b}f(x)
e^{i\omega x}\mathrm{d}x$, developed in \cite{Asheim2014}, we extend
the idea to the current setting. Specifically, let
$\{q_n^{\omega}\}_{n=0}^{\infty}$ be the sequence of orthogonal
polynomials with respect to the function $(g\circ
\phi_{\omega})(x)-\rho_0/2$, i.e.,
\begin{align}\label{def:qn}
\int_{a}^{b} q_n^{\omega}(x) x^j \left[ (g\circ
\phi_{\omega})(x)-\frac{\rho_0}{2} \right] \mathrm{d}x = 0, \quad
j=0,\ldots,n-1.
\end{align}
Notice that $(g\circ \phi_{\omega})(x)-\rho_0/2$ may be a
sign-changing function on the interval $[a,b]$, the existence of
$\{q_n^{\omega}\}_{n=0}^{\infty}$ can not be guaranteed anymore.
However, once $q_n^{\omega}(x)$ exists, we can construct a Gaussian
quadrature of the form
\begin{equation}\label{eq:IoGauss}
I_{O}[f] = \sum_{k=1}^{n} w_k f(x_k) + \mathcal{R}_n^{O}[f],
\end{equation}
where the quadrature nodes $\{x_k\}_{k=1}^{n}$ are the zeros of
$q_n^{\omega}(x)$ and $\mathcal{R}_n^{O}[f]=0$ for
$f\in\mathcal{P}_{2n-1}$. To gain some insight into the Gaussian
quadrature \eqref{eq:IoGauss}, we consider the example
$g(x)=\ln(x+4)$, $\phi_{\omega}(x)=\cos(\omega x)$ and we restrict
our attention to the case where $n$ is an integer multiple of 4,
i.e., $n=4,8,\ldots$. For this example, it is easily checked from
\eqref{def:rho} that $\rho_0=2\ln(2+\sqrt{15}/2)$. The zeros of
$q_n^{\omega}(x)$ are computed by using steps similar to the
computation of the zeros of $p_n^{\omega}(x)$ in subsection
\ref{sec:OrthPolyI}. In Figure \ref{fig:Trajectory} we display the
trajectories of the zeros of $q_n^{\omega}(x)$ as a function of
$\omega$ for $\omega\in[20,100]$. Clearly, we see that the zeros of
$q_n^{\omega}(x)$ are distributed in the complex plane and half the
zeros converge towards to the left endpoint and half the zeros
converge towards to the right endpoint as $\omega\rightarrow\infty$.
To clarify the rate of convergence of these zeros to the endpoints,
we further display the absolute errors of these zeros with the
corresponding endpoint scaled by $\omega$ and the computations were
performed by using the multiprecision computing toolbox available at
\texttt{www.advanpix.com} with 20-digit arithmetic. Numerical
results are illustrated in Figure \ref{fig:Trajectory} and we can
see that the zeros of $q_n^{\omega}(x)$ converge to the
corresponding endpoint at the rate $O(\omega^{-1})$ as
$\omega\rightarrow\infty$. We also examined more examples, including
$g(x)=e^x$, $(x+2)^3$, $1/(4-x)$, $1/(1+x^2)$, and the computed
results are quite similar to those displayed in Figures
\ref{fig:Trajectory} and are thus omitted. We summarize our
experimental observations when $n$ is an integer multiple of 4 as
follows:
\begin{itemize}
\item The zeros of $q_n^{\omega}(x)$ are distributed in the complex
plane and half the zeros converge toward the left endpoint and
half the zeros converge toward the right endpoint as
$\omega\rightarrow\infty$.

\item The trajectories of the zeros of $q_n^{\omega}(x)$ are
symmetric with respect to both the real and imaginary axes whenever
$(g\circ\phi_{\omega})(x)$ is even or odd.

\item The rates of convergence of the zeros of $q_n^{\omega}(x)$ to one of the endpoints are $O(\omega^{-1})$ as $\omega\rightarrow\infty$.
\end{itemize}
We remark that the behavior of the zeros of $q_n^{\omega}(x)$ is
quite similar to that of the zeros of the polynomials orthogonal
with respect to the function $e^{i\omega x}$ (see \cite{Asheim2014}
for details). Based on the above observations, we now prove the
asymptotic error estimate of the Gaussian quadrature rule for
computing $I_{O}[f]$ for large $\omega$ under the hypotheses on the
asymptotic behavior of the quadrature nodes.

\begin{figure}
\centering
\includegraphics[width=5.6cm,height=5.5cm]{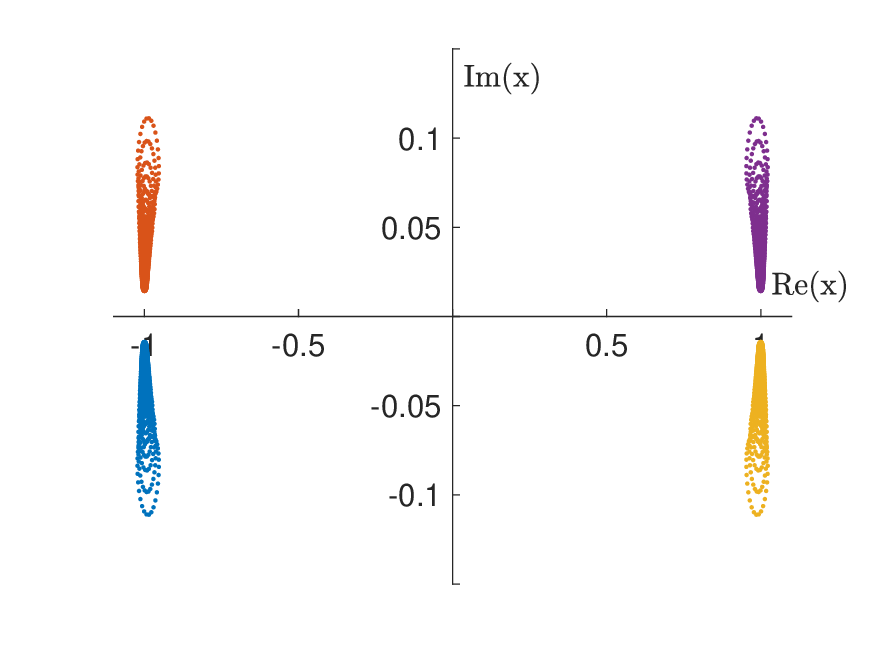}~~~~
\includegraphics[width=5.6cm,height=5.5cm]{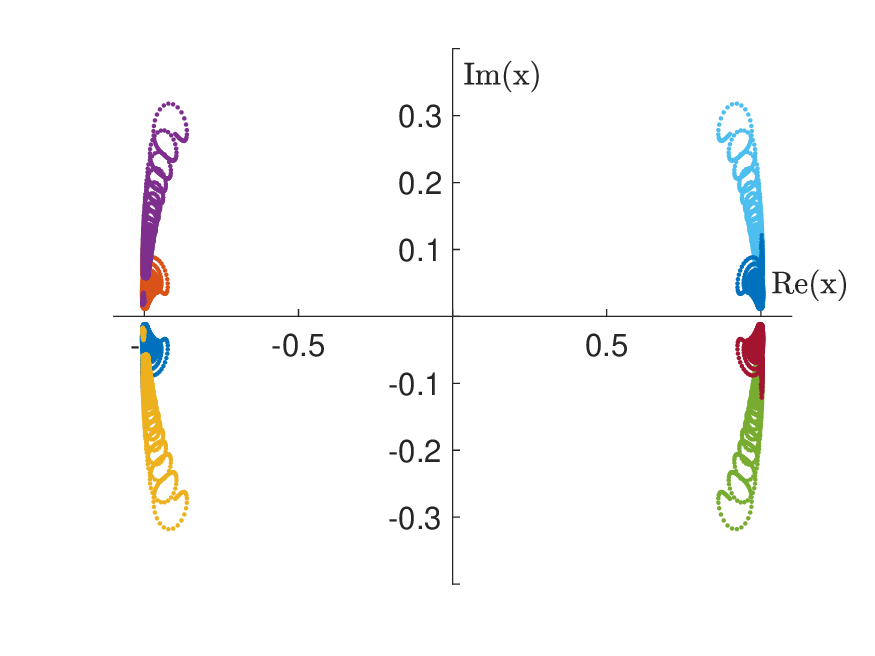}
\includegraphics[width=5.6cm,height=5.5cm]{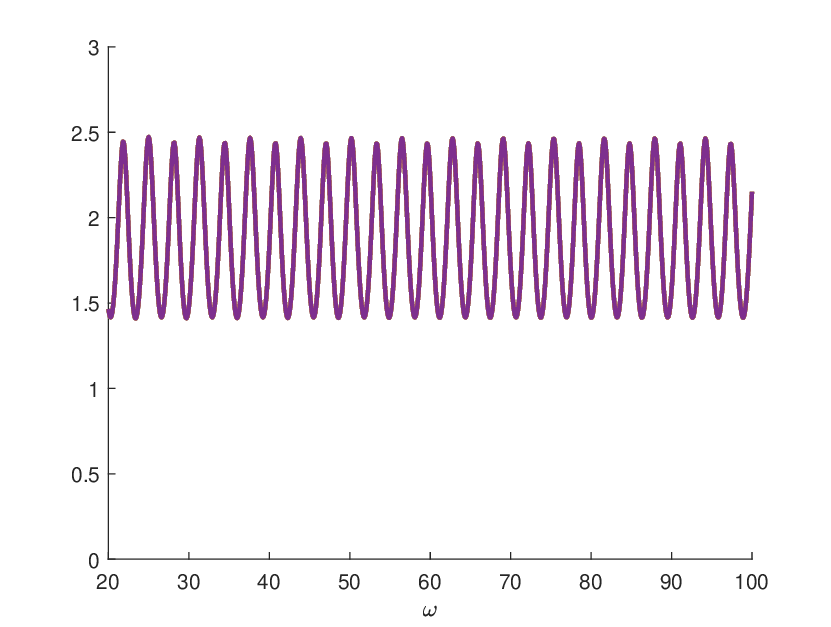}~~~~
\includegraphics[width=5.6cm,height=5.5cm]{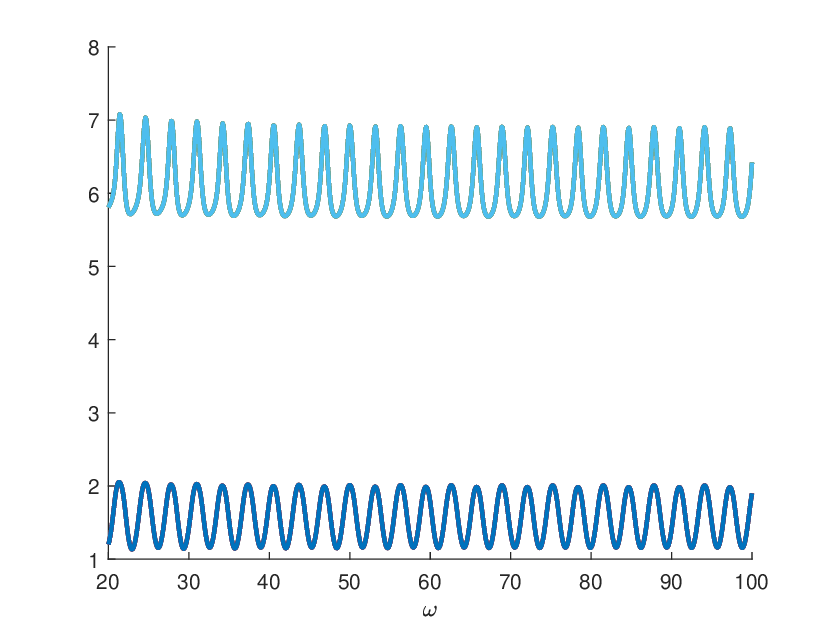}
\caption{The trajectories of the zeros of $q_4^{\omega}(x)$ and
$q_8^{\omega}(x)$ as a function of $\omega$ (top) and the absolute
errors between the zeros and the corresponding endpoint scaled by
$\omega$ (bottom). Here $g(x)=\ln(x+4)$,
$\phi_{\omega}(x)=\cos(\omega x)$, $[a,b]=[-1,1]$ and $\omega$
ranges from $20$ to $100$.}\label{fig:Trajectory}
\end{figure}

\begin{theorem}\label{thm:AsmErrorGauss}
Suppose that $f$ is analytic in a region containing $[a,b]$. Let $n$
be an even positive integer and let $\{x_k,w_k\}_{k=1}^{n}$ be the
nodes and weights of the Gaussian quadrature rule in
\eqref{eq:IoGauss}. If the quadrature nodes $\{x_k\}_{k=1}^{n}$ can
be splitted into two groups $\{\check{x}_k,\hat{x}_k\}_{k=1}^{n/2}$
which satisfy
\begin{align}
\check{x}_k = a + O(\omega^{-1}), \quad  \hat{x}_k = b +
O(\omega^{-1}), \quad  k=1,\ldots,\frac{n}{2}.
\end{align}
Then,
\begin{equation}
I_{O}[f] - \sum_{k=1}^{n}w_kf(x_k) = O(\omega^{-n-1}), \quad
\omega\gg1.
\end{equation}
\end{theorem}
\begin{proof}
We follow the line as the proof of Theorem 4.1 in \cite{Asheim2014}
and we only prove the case of $\phi_{\omega}(x)=\sin(\omega x)$
since the proof of the case $\phi_{\omega}(x)=\cos(\omega x)$ is
similar. Let $h_{2n-1}(x)$ be the Hermite interpolation polynomial
of degree $2n-1$ which interpolates $f$ and its derivatives up to
order $n-1$ at both endpoints $\{a,b\}$. By the Hermite's integral
formula \cite[Theorem~3.6.1]{Davis1975} we have that
$f(x) = h_{2n-1}(x) + \zeta(x)(x-a)^n(x-b)^n$, 
where $\zeta(x)$ is analytic inside the analyticity region of $f$.
Recalling the Gaussian quadrature rule in \eqref{eq:IoGauss}
satisfies $\mathcal{R}_n^{O}[f]=0$ for $f\in\mathcal{P}_{2n-1}$, we
obtain that
\begin{align}\label{eq:Remainder}
I_{O}[f] - \sum_{k=1}^{n} w_k f(x_k) &= \int_{a}^{b} \zeta(x)(x-a)^n(x-b)^n \left[(g\circ \phi_{\omega})(x)-\frac{\rho_0}{2}\right] \mathrm{d}x \nonumber \\
&~~~~~  -\sum_{k=1}^{n} w_k \zeta(x_k)(x_k-a)^n(x_k-b)^n.
\end{align}
For the first term on the right hand side of \eqref{eq:Remainder},
note that $\{a,b\}$ are the $n$ multiple zeros of
$\zeta(x)(x-a)^n(x-b)^n$. Invoking the asymptotic expansion of the
composite highly oscillatory integrals in \eqref{eq:asymptotic}, we
deduce that the first term behaves like $O(\omega^{-n-1})$ as
$\omega\rightarrow\infty$. Next, we consider the asymptotic estimate
of the second term on the right hand side of \eqref{eq:Remainder}.
Since the quadrature nodes satisfy $x_k=a+O(\omega^{-1})$ or
$x_k=b+O(\omega^{-1})$, we deduce immediately that
$(x_k-a)^n(x_k-b)^n = O(\omega^{-n})$ for all $k=1,\ldots,n$. Now we
consider the asymptotic estimate of the quadrature weight
$\{w_k\}_{k=1}^{n}$. Recall that the Gaussian quadrature is an
interpolatory quadrature rule, we infer that
$$
w_k = \int_{a}^b \ell_k(x) \left[ (g\circ \phi_{\omega})(x) -
\frac{\rho_0}{2} \right] \mathrm{d}x,
$$
where $\ell_k(x)$ is the $k$th Lagrange basis polynomial associated
with the nodes $\{x_k\}_{k=1}^{n}$. Furthermore, using the
asymptotic expansion in \eqref{eq:asymptotic} to $w_k$ and noting
that $\ell_k(x)$ is a polynomial of degree $n-1$, we obtain that
\begin{align}
w_k &= ~\sum_{j=0}^{n/2-1} \frac{(-1)^j}{\omega^{2j+1}} \left[ \ell_k^{(2j)}(b) U_k(b) - \ell_k^{(2j)}(a) U_k(a) \right] \nonumber\\
&~~~~~~ + \sum_{j=0}^{n/2-1} \frac{(-1)^j}{\omega^{2j+2}} \left[
\ell_k^{(2j+1)}(b) V_k(b) - \ell_k^{(2j+1)}(a) V_k(a) \right],
\nonumber
\end{align}
where $U_k(x)$ and $V_k(x)$ are defined as in \eqref{eq:UV} and
$U_k(x)=O(1)$ and $V_k(x)=O(1)$ as $\omega\rightarrow\infty$. In the
following we shall consider the estimates of $\ell_k^{(j)}(a)$ and
$\ell_k^{(j)}(b)$ for $k=1,\ldots,n$ and $j=0,\ldots,n-1$ as
$\omega\rightarrow\infty$. Assuming that $x_k = a + O(\omega^{-1})$
for $k=1,\ldots,n/2$ and $x_k = b + O(\omega^{-1})$ for
$k=n/2+1,\ldots,n$ and letting $\Lambda_1$ and $\Lambda_2$ denote
the sets $\{0,\ldots,n/2-1\}$ and $\{n/2,\ldots,n-1\}$,
respectively. From the definition of $\ell_k(x)$ it follows that
\begin{equation}
\ell_k(x) = \frac{\displaystyle\prod_{j=1,j\neq
k}^{n}(x-x_j)}{\displaystyle\prod_{j=1,j\neq k}^{n}(x_k-x_j)}, \quad
\ell_k{'}(x) = \frac{\displaystyle\sum_{\nu=1,\nu\neq
k}^{n}\prod_{\mu=1,\mu\neq
k,\nu}^{n}(x-x_{\mu})}{\displaystyle\prod_{j=1,j\neq k}^{n}
(x_k-x_j)}, \nonumber
\end{equation}
and hence, by direct calculation, we find that $\ell_k(a)=O(1)$ and
$\ell_k{'}(a)=O(\omega)$ for $k=1,\ldots,n/2$ and
$\ell_k(a)=O(\omega^{-1})$ and $\ell_k{'}(a)=O(1)$ for
$k=n/2+1,\ldots,n$, and $\ell_k(b)=O(\omega^{-1})$ and
$\ell_k{'}(b)=O(1)$ for $k=1,\ldots,n/2$ and $\ell_k(b)=O(1)$ and
$\ell_k{'}(b)=O(\omega)$ for $k=n/2+1,\ldots,n$. Further, it can be
shown by induction that, for $k=1,\ldots,n/2$,
\begin{align}
\ell_k^{(j)}(a) = \left\{
\begin{array}{ll}
O(\omega^j),       & j\in\Lambda_1, \\[6pt]
O(\omega^{n/2-1}), & j\in\Lambda_2,
\end{array}
\right. \quad  \ell_k^{(j)}(b) = \left\{
\begin{array}{ll}
O(\omega^{j-1}),   & j\in\Lambda_1, \\[6pt]
O(\omega^{n/2-1}), & j\in\Lambda_2,
\end{array}
\right. \nonumber
\end{align}
and for $k=n/2+1,\ldots,n$,
\begin{align}
\ell_k^{(j)}(a)=
\left\{
\begin{array}{ll}
O(\omega^{j-1}),   & j\in\Lambda_1, \\[6pt]
O(\omega^{n/2-1}), & j\in\Lambda_2,
\end{array}
\right. \quad \ell_k^{(j)}(b)= \left\{
\begin{array}{ll}
O(\omega^{j}),     & j\in\Lambda_1, \\[6pt]
O(\omega^{n/2-1}), & j\in\Lambda_2,
\end{array}
\right. \nonumber
\end{align}
and thus we can deduce that $w_k = O(\omega^{-1})$ for each
$k=1,\ldots,n$. Combining this with the fact $(x_k-a)^n(x_k-b)^n =
O(\omega^{-n})$, we deduce that the second term on the right hand
side of \eqref{eq:Remainder} also behaves like $O(\omega^{-n-1})$ as
$\omega\rightarrow\infty$. Thus, we conclude that both terms on the
right hand side of \eqref{eq:Remainder} are $O(\omega^{-n-1})$,
which completes the proof.
\end{proof}

To show the sharpness of our asymptotic error estimates derived in
Theorem \ref{thm:AsmErrorGauss}, we consider two examples:
$f(x)=1/(x+2),~(g\circ\phi_{\omega})(x)=\exp(\sin(\omega x))$ and
$f(x)=1/(x^2+1),~(g\circ \phi_{\omega})(x)=\ln(4+\cos(\omega x))$.
In each example we evaluate the integral $I_{S}[f]$ directly and
compute the integral $I_{O}[f]$ by applying the Gaussian quadrature
rule in \eqref{eq:IoGauss}. In Figure \ref{errorsemilog1} we display
the absolute errors of the $n$-point Gaussian quadrature rule for
$n=4$ and $n=8$. It is easily seen that the error of the Gaussian
quadrature rule behaves like $O(\omega^{-5})$ for $n=4$ and
$O(\omega^{-9})$ for $n=8$, and these results are consistent with
our theoretical analysis.

\begin{figure}
\centering
\includegraphics[width=5.6cm,height=5.5cm]{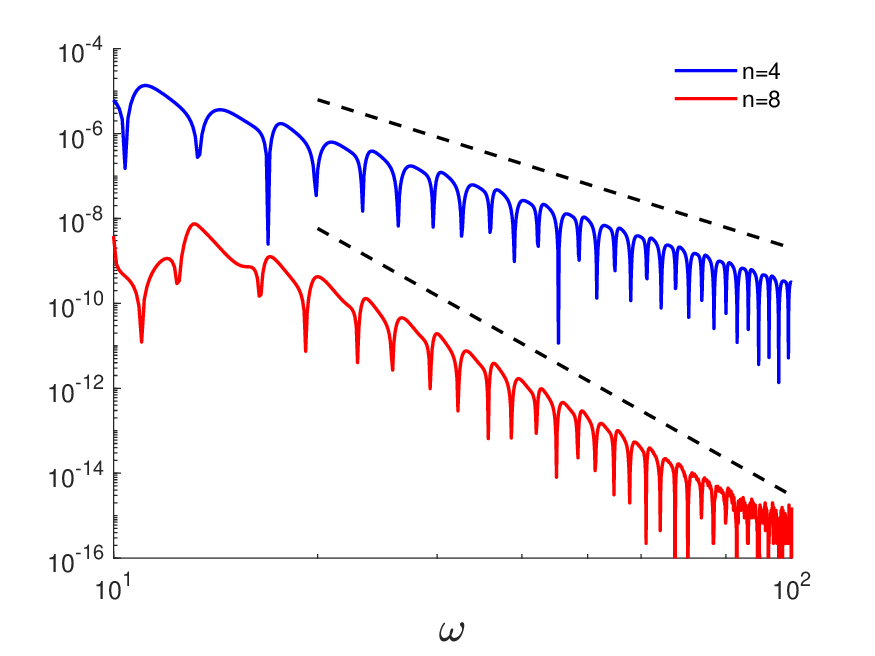}~~~
\includegraphics[width=5.6cm,height=5.5cm]{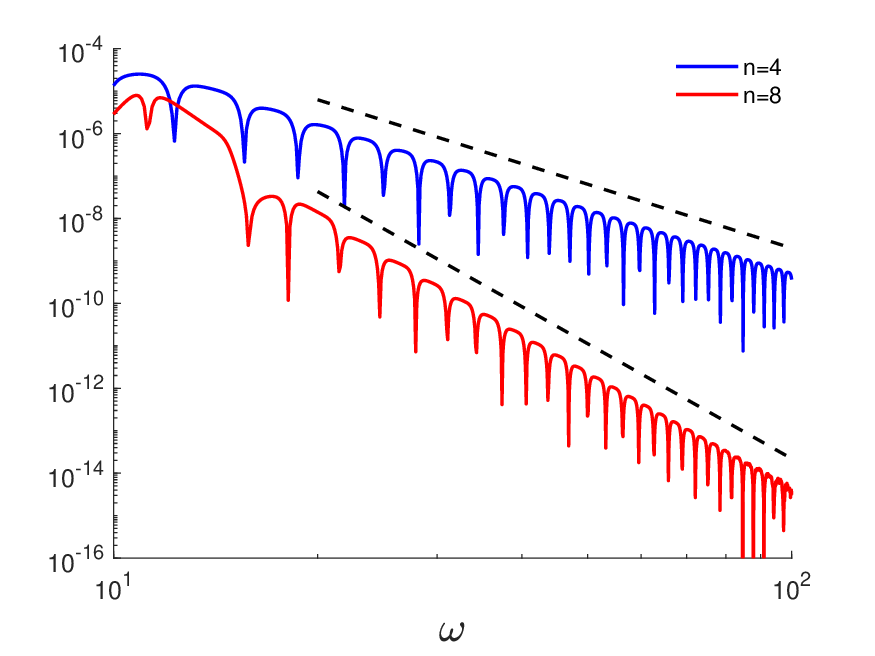}
\caption{Absolute errors of the Gaussian quadrature rule for
computing the integrals $\int_{-1}^{1} \exp(\sin(\omega x))/(x+2)
\mathrm{d}x$ (left) and $\int_{-1}^{1} \ln(4+\cos(\omega x))/(x^2+1)
\mathrm{d}x$ (right) as a function of $\omega$. The dash lines show
$O(\omega^{-5})$ (top) and $O(\omega^{-9})$ (bottom).}
\label{errorsemilog1}
\end{figure}

Finally, we consider comparing the performance of the Gaussian
quadrature rule in \eqref{eq:IoGauss} with the Filon-type methods
developed in \cite{Iserles2011}. Let $a=c_1<c_2<\cdots<c_{\nu}=b$ be
a set of points and let $\mu_1,\mu_2,\cdots,\mu_\nu\in \mathbb{N}$
be the corresponding multiplicities. Let $\psi$ be a polynomial of
degree $N = \sum_{k=1}^{\nu}\mu_k-1$ which satisfies
$$
\psi^{(j)}(c_k) = f^{(j)}(c_k), \quad j=0,1,\cdots,\mu_k-1, \quad
k=1,2,\cdots,\nu.
$$
The Filon-type method for computing $I[f]$ is defined by (see
\cite[Equation~(3.11)]{Iserles2011})
\begin{align}\label{eq:FilonType}
Q_r^F[f] 
&= \frac{\rho_0}{2} \int_a^b f(x) \mathrm{d}x + \sum_{k=0}^{\lfloor N/2 \rfloor} \frac{(-1)^k}{\omega^{2k+1}} \left[ \psi^{(2k)}(b)U_k(b) - \psi^{(2k)}(a)U_k(a) \right] \nonumber \\
&~~~~~ + \sum_{k=0}^{\lfloor (N-1)/2 \rfloor}
\frac{(-1)^k}{\omega^{2k+2}} \left[\psi^{(2k+1)}(b)V_k(b) -
\psi^{(2k+1)}(a)V_k(a) \right],
\end{align}
where $r=\min\{\mu_1, \mu_\nu\}$. 
Furthermore, it was proved in \cite[Theorem 2]{Iserles2011} that the
asymptotic error estimate of $Q_r^F[f]$ is $O(\omega^{-r-1})$ as
$\omega\rightarrow\infty$. We first consider the test functions:
$g(x)=e^{x},~1/(4-x),~\sin(x)$ and we choose
$f(x)=\sin(x),~\phi_{\omega}(x)=\sin(\omega x)$ and $[a,b]=[-1,1]$.
In our comparison, we choose $n=4$ in the Gaussian quadrature
\eqref{eq:IoGauss} and $c_1=-1,~c_2=1$ and $\mu_1=2,~\mu_2=2$ in the
Filon-type method \eqref{eq:FilonType} and thus both methods are
implemented with the same number of function evaluations. Figure
\ref{compare1} illustrates the absolute errors of the Gaussian
quadrature rule scaled by $\omega^{5}$ and the absolute errors of
the Filon-type method scaled by $\omega^{3}$. Clearly, we see that
the asymptotic order of the Gaussian quadrature rule is higher than
that of the Filon-type method. Next, we consider another set of test
functions: $g(x) = e^{x},~1-x,~\ln(4+x)$ and we choose
$f(x)=1/(x+2)$ and $\phi_{\omega}(x)=\cos(\omega x)$. Figure
\ref{compare2} illustrates the absolute errors of the Gaussian
quadrature rule scaled by $\omega^{5}$ and the absolute errors of
the Filon-type method scaled by $\omega^{3}$. We observe again that
the asymptotic order of the Gaussian quadrature rule is higher than
that of the Filon-type method.

\begin{figure}
\centering \subfigure[]{
\includegraphics[width=4.cm,height=4.2cm]{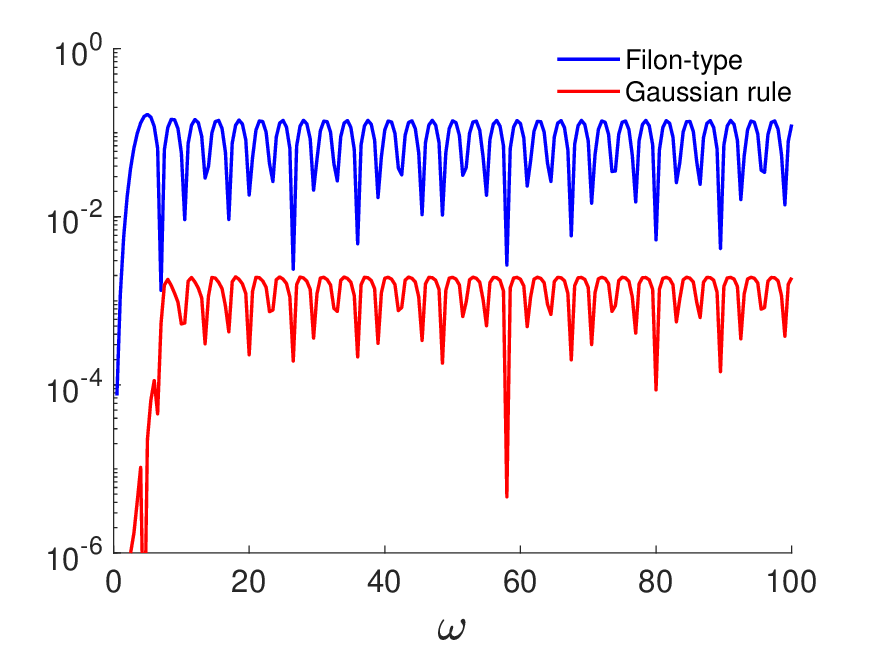}}~~
\subfigure[]{
\includegraphics[width=4.cm,height=4.2cm]{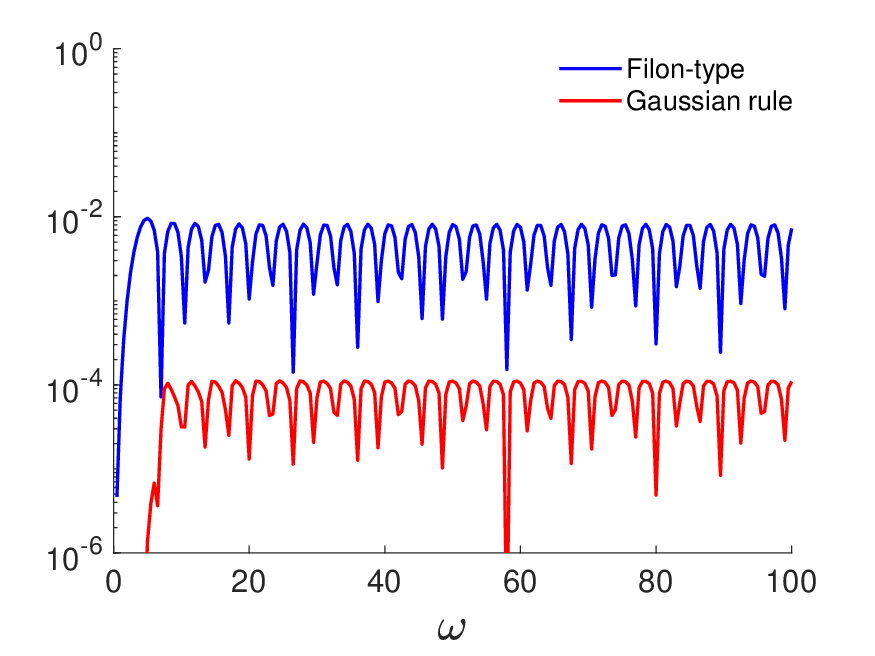}}~~
\subfigure[]{
\includegraphics[width=4.cm,height=4.2cm]{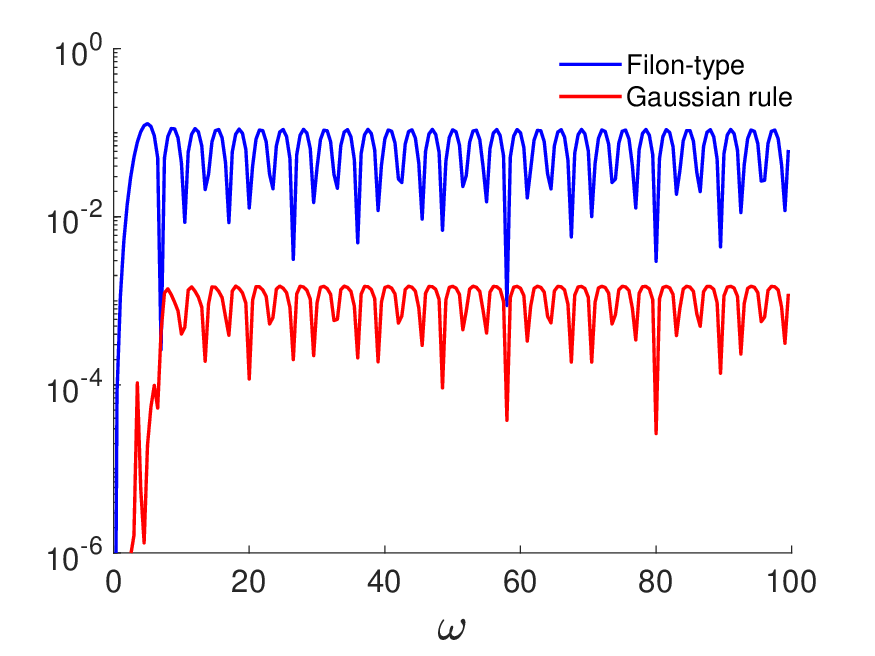}}
\caption{Absolute errors of the Gaussian quadrature rule scaled by
$\omega^{5}$ and the absolute errors of the Filon-type method scaled
by $\omega^{3}$ for the integrals: (a) $\int_{-1}^{1} \sin(x)
e^{\sin(\omega x)} \mathrm{d}x$, (b) $\int_{-1}^{1}
\sin(x)/(4-\sin(\omega x)) \mathrm{d}x$ and (c) $\int_{-1}^{1}
\sin(x) \sin(\sin(\omega x)) \mathrm{d}x$.} \label{compare1}
\end{figure}

\begin{figure}
\centering \subfigure[]{
\includegraphics[width=4.1cm,height=4.2cm]{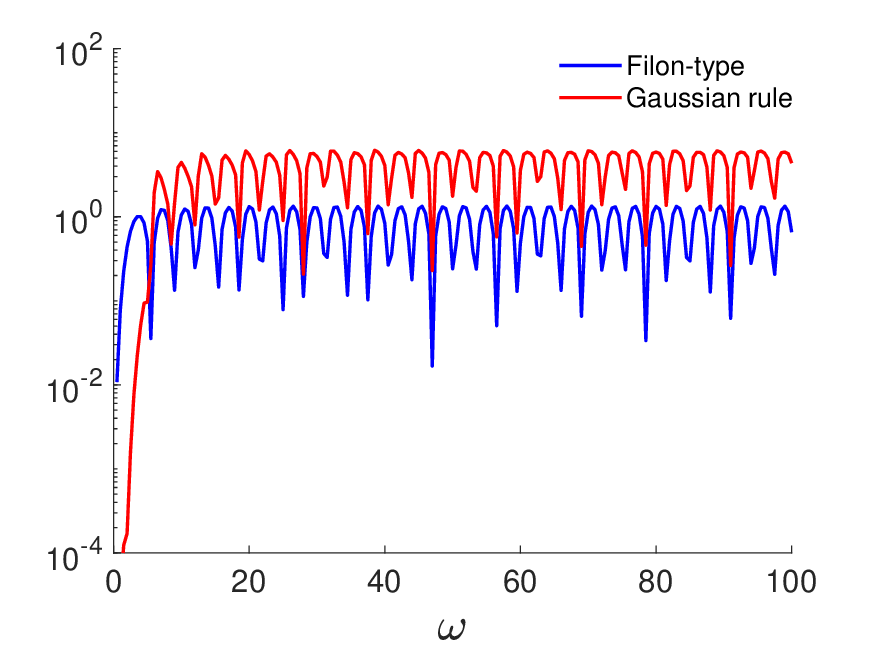}}~~
\subfigure[]{
\includegraphics[width=4.1cm,height=4.2cm]{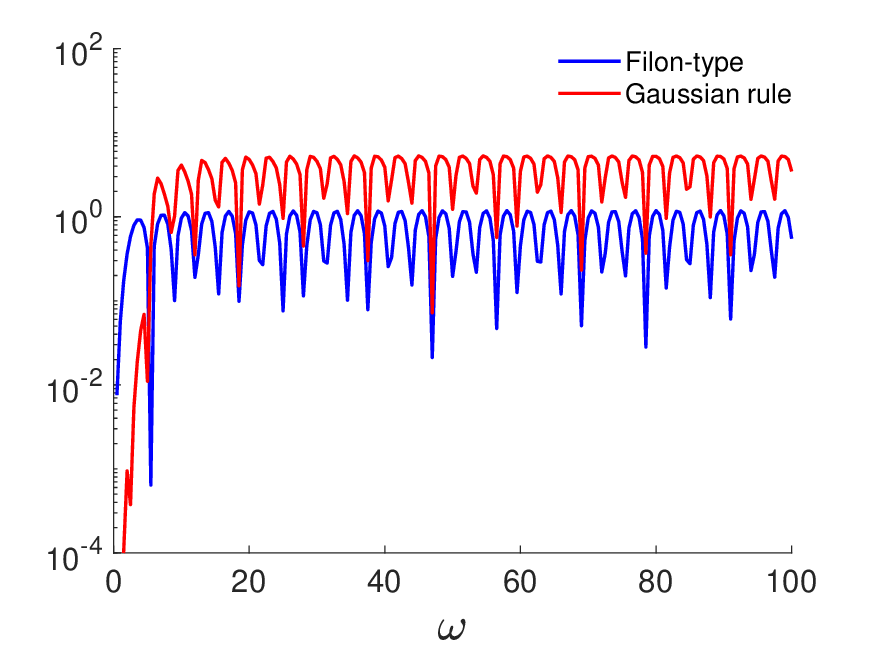}}~~
\subfigure[]{
\includegraphics[width=4.1cm,height=4.2cm]{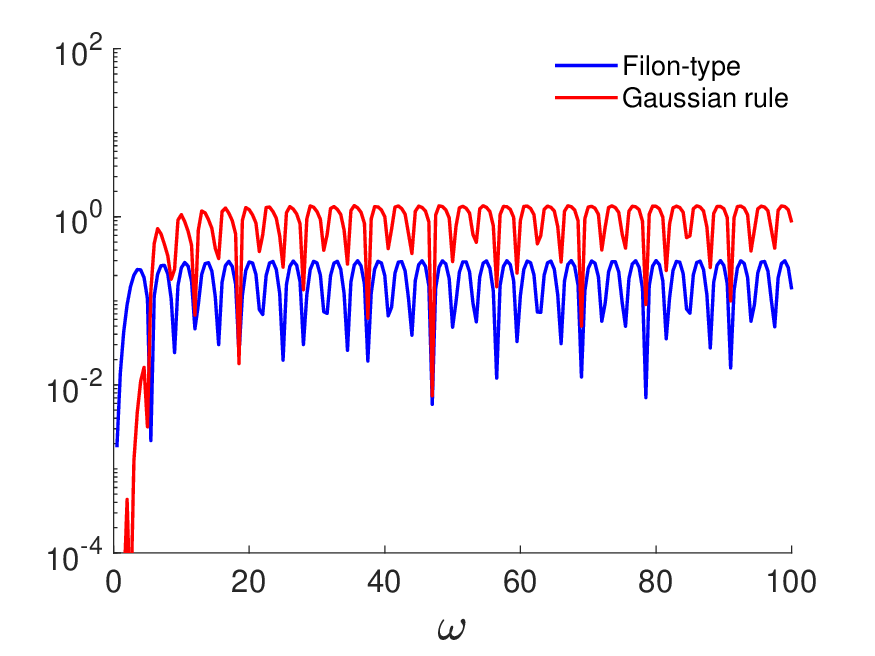}}
\caption{Absolute errors of the Gaussian quadrature rule scaled by
$\omega^5$ and the absolute error of the Filon-type method scaled by
$\omega^{3}$ for the integrals: (a) $\int_{-1}^{1} e^{\cos(\omega
x)}/(x+2) \mathrm{d}x$, (b) $\int_{-1}^{1} (1-\cos(\omega x))/(x+2)
\mathrm{d}x$ and (c) $\int_{-1}^{1} \ln(4+\cos(\omega x))/(x+2)
\mathrm{d}x$.} \label{compare2}
\end{figure}

\section{Concluding remarks}\label{sec:conclusion}
In this paper, we proposed two Gaussian quadrature rules for
computing composite highly oscillatory integrals of the form
\eqref{CHOI}. The idea of the first Gaussian quadrature rule is to
treat the oscillatory part $(g\circ \phi_{\omega})(x)$ as the weight
function and then construct a Gaussian quadrature rule for computing
the integrals. We showed that this quadrature is guaranteed by the
classical theories of orthogonal polynomials and its nodes and
weights can be computed efficiently using tools of numerical linear
algebra. We proved that the quadrature nodes have an interesting
property that they converge to the Legendre points at the rate
$O(\omega^{-1})$ as $\omega\rightarrow\infty$. Moreover, we
presented a rigorous convergence analysis of this Gaussian
quadrature rule and showed that its convergence rate depends solely
on the regularity of $f(x)$. To develop a method for computing
\eqref{CHOI} such that its accuracy improves as $\omega$ increases,
we further proposed the second Gaussian quadrature rule based on the
asymptotic analysis of the integrals \eqref{CHOI} developed in
\cite{Iserles2011}. In this case, however, the Gaussian quadrature
rule is constructed with respect to a sign-changing function and its
existence can not be guaranteed. We explored numerically the
behavior of the quadrature nodes and presented some experimental
observations, including the trajectories of the zeros and the
convergence rate of the zeros to the endpoints as
$\omega\rightarrow\infty$. Based on these observations, we showed
that the asymptotic error estimate of an $n$-point Gaussian
quadrature rule is $O(\omega^{-n-1})$ and thus the accuracy of the
Gaussian quadrature rule improves rapidly as $\omega$ increases.

Finally, we point out that the properties of orthogonal polynomials
with respect to the complex exponential function $e^{i\omega x}$
were extensively studied in \cite{Celsus2021} and it was proved that
the orthogonal polynomials with respect to $e^{i\omega x}$ with even
degree always exist for $\omega>0$ and their zeros tend to one of
the endpoints at the rate $O(\omega^{-1})$ as
$\omega\rightarrow\infty$. Note that similar properties on the
behavior of the zeros of $q_n^{\omega}(x)$ with $n$ being an integer
multiple of four were observed, it is still unclear whether the
proof can be extended to the case of $q_n^{\omega}(x)$. Another
challenging problem is the design of efficient algorithms for
Gaussian quadrature rules developed in this work. Note that the
algorithms presented in subsection \ref{sec:OrthPolyI} for the
computation of the modified Chebyshev moments
$\{\nu_j\}_{j=0}^{2n-1}$ are still not satisfactory and the
algorithm for the second Gaussian quadrature rule in section
\ref{sec:SecondGauss} requires high-precision arithmetic for
moderate and large values of $n$. On the other hand, we point out
that the three-term recurrence coefficients of $\{p_n^{\omega}(x)\}$
can actually be calculated by using the modified Chebyshev algorithm
\cite[Section~2.1.7]{Gautschi2004}, and therefore the first Gaussian
quadrature rule can also be achieved by the well-known Golub-Welsch
algorithm (see \cite[Section~3.1]{Gautschi2004}). For the second
Gaussian quadrature rule, however, the modified Chebyshev algorithm
for computing the three-term recurrence coefficients of
$\{q_n^{\omega}(x)\}$ will suffer from numerical instability due to
the fact that the polynomials $q_n^{\omega}(x)$ are orthogonal with
respect to a sign-changing function. We leave these issues for
future work.

\section*{Acknowledgements}
The second author wishes to thank Prof. Shuhuang Xiang for helpful
discussion on the topic of this work. The authors also would like to
thank two anonymous referees for their valuable comments on this
work.

\bibliographystyle{amsplain}

\begin{thebibliography}{9}
\bibitem{Asheim2014}
A. Asheim, A. Dea{\~n}o, D. Huybrechs and H.-Y. Wang, \emph{A
Gaussian quadrature rule for oscillatory integrals on a bounded
interval}, Dis. Contin. Dyn. Sys. A, 34(3):883-901, 2014.


\bibitem{Beylkin2016}
G. Beylkin and L. Monz{\'o}n, \emph{Efficient representation and
accurate evaluation of oscillatory integrals and functions}, Dis.
Contin. Dyn. Sys., 36(8):4077-4100, 2016.

\bibitem{JPBoyd2000}
J. P. Boyd, \emph{Chebyshev and Fourier Spectral Methods}, Dover
Publications, New York, 2000.


\bibitem{Celsus2021}
A. F. Celsus, A. Dea\~{n}o, D. Huybrechs and A. Iserles, \emph{The
kissing polynomials and their Hankel determinants}, Trans. Math.
Appl., 6(1):1-66, 2022.



\bibitem{Condon2009a}
M. Condon, A. Dea{\~n}o and A. Iserles, \emph{On highly oscillatory
problems arising in electronic engineering}, ESIAM: M2AN,
43(4):785-804, 2009.


\bibitem{Condon2009b}
M. Condon, A. Dea{\~n}o, A. Iserles, K. Maczy{\'n}ski and T. Xu,
\emph{On numerical methods for highly oscillatory problems in
circuit simulation}, COMPEL, 28(6):1607-1618, 2009.

\bibitem{Chihara1978}
T. S. Chihara, \emph{An Introduction to Orthogonal Polynomials},
Gordon and Breach, Science Publishers Inc., New York, 1978.

\bibitem{Dautb2005}
E. Dautbegovic, M. Condon and C. Brennan, \emph{An efficient
nonlinear circuit simulation technique}, IEEE Trans. Microwave
Theory Tech., 53(2):548-555, 2005.


\bibitem{Davis1975}
P. J. Davis, \emph{Interpolation and Approximation}, Dover
Publications, New York, 1975.



\bibitem{DHIbook2017}
A. Dea{\~n}o, D. Huybrechs and A. Iserles, \emph{Computing Highly
Oscillatory Integrals}, SIAM, Philadelphia, 2018.


\bibitem{Gautschi2004}
W. Gautschi, \emph{Orthogonal Polynomials: Computation and
Approximation}, Oxford University Press, New York, 2004.


\bibitem{huybrechs2006}
D. Huybrechs and S. Vandewalle, \emph{On the evaluation of highly
oscillatory integrals by analytic continuation}, SIAM J. Numer.
Anal., 44(3):1026-1048, 2006.


\bibitem{Iserles2011}
A. Iserles and D. Levin, \emph{Asymptotic expansion and quadrature
of composite highly oscillatory integrals}, Math. Comp.,
80(273):279-296, 2011.


\bibitem{Iserles2005}
A. Iserles and S. P. N{\o}rsett, \emph{Efficient quadrature of
highly oscillatory integrals using derivatives}, Proc. Roy. Soc. A,
461(2057):1383-1399, 2005.


\bibitem{Levin1982}
D. Levin, \emph{Procedures for computing one- and two-dimensional
integrals of functions with rapid irregular oscillations}, Math.
Comp., 38(158):531-538, 1982.


\bibitem{Olver2006}
S. Olver, \emph{Moment-free numerical integration of highly
oscillatory functions}, IMA J. Numer. Anal., 26(2):213-227, 2006.



\bibitem{Trefethen2020}
L. N. Trefethen, \emph{Approximation Theory and Approximation
Practice}, Extended Edition, SIAM, Philadelphia, 2019.
\end{thebibliography}

\end{document}